\documentclass[11pt, reqno]{amsart}
\usepackage{graphicx}              % to include figures
\usepackage{amsmath}              % great math stuff
\usepackage{amsfonts}              % for blackboard bold, etc
\usepackage{amsthm}                % better theorem environments
\usepackage{color}
\usepackage{tabulary}
\usepackage{caption}
\usepackage{subcaption}
\usepackage{amssymb}
\usepackage{todonotes}
\theoremstyle{plain}
\usepackage{mathtools}
\usepackage{soul}
\theoremstyle{plain}

\newtheorem{theorem}{Theorem}[section]

\newtheorem{proposition}[theorem]{Proposition}
\newtheorem{lemma}[theorem]{Lemma}

\newtheorem{corollary}[theorem]{Corollary}
\theoremstyle{definition}
\newtheorem{definition}[theorem]{Definition}

\newtheorem{remark}[theorem]{Remark}

\def\s{\mathsf}

\def\d{\delta}
\let\s\mathsf
\let\Cal\mathcal
\def\Frac#1#2{\frac{\textstyle #1}{\textstyle #2}}

\newcommand{\nc}{\newcommand}
\nc{\dmo}{\DeclareMathOperator}

\nc{\Q}{\mathbb{Q}}
\nc{\R}{\mathbb{R}}
\nc{\Z}{\mathbb{Z}}
\nc{\N}{\mathbb{N}}
\nc{\C}{\mathbb{C}}
\nc{\cS}{\mathcal{S}}
\nc{\iso}{\cong}
\dmo{\Mod}{Mod}
\dmo{\Diff}{Diff}
\dmo{\Homeo}{Homeo}
\dmo{\dist}{dist}
\dmo\BDiff{BDiff}
\dmo\SO{SO}
\dmo\slide{sl}
\dmo\im{im}
\dmo\id{id}
\dmo\Fix{Fix}
\dmo\Out{Out}
\dmo{\T}{\mathcal{T}}
\dmo{\Te}{\mathcal{T}^{\epsilon}}
\dmo{\M}{\mathcal{M}}
\dmo{\Me}{\mathcal{M}^{\epsilon}}
\def\GAF{{\overline G_f}}
\def\pfg{\partial_fG}
\def\act{{\curvearrowright}}
\def\La{{\bf\Lambda}}
\def\Gf{{\overline G}_f}
\def\ch{{\mathcal H}}
\def\ve{\varepsilon}
\def\psf{\partial^f}
\def\ga{\gamma}

\def\psf{\partial^f}
\def\plf{\partial_f}

\def\psm{\partial^{\c M}}

\def\pfg{\partial_fG}
\def\pmg{\partial_{\M}G}
\def\Gm{{\overline G}_{\M}}
\def\Gam{G\sqcup\pmg}
\def\Gb{{\overline G}_B}

\let\c\mathcal

\renewcommand{\epsilon}{\varepsilon}
\nc{\coloneq}{\mathrel{\mathop:}\mkern-1.2mu=}
\nc{\margin}[1]{\marginpar{\scriptsize #1}}
\nc{\para}[1]{\bigskip\noindent\textbf{#1}}

\def\bx{$\hfill\square$}
\def\foral{\forall\hspace*{0.5mm}}
\def\exist{\exists\hspace*{0.7mm}}

\title [Martin and Floyd boundaries]{Martin boundary covers Floyd boundary}
\author[I. Gekhtman, V. Gerasimov, L. Potyagailo, W. Yang]{Ilya Gekhtman, Victor Gerasimov, Leonid Potyagailo, and Wenyuan Yang}
 
\begin{document}

 \begin{abstract}
 For finitely supported random walks on finitely generated groups $G$ we prove that the identity map on $G$ extends to a continuous {equivariant} surjection from the Martin boundary to the Floyd boundary, with preimages of conical points being singletons. This yields new results for relatively hyperbolic groups. Our key estimate relates the Green and Floyd metrics, generalizing results of Ancona for random walks on hyperbolic groups and of Karlsson for quasigeodesics.
 {We then apply}  these {techniques} to obtain some results concerning the harmonic measure on the {limit sets of geometrically finite isometry groups of} Gromov hyperbolic spaces.   .
 \end{abstract}
 \maketitle
 \section{Introduction}
It is a common thread in geometric group theory to relate asymptotic properties of random walks on a group to the dynamics of its action on some geometric boundary.   The Green metric $d_{\mathcal{G}}(.,.)$ associated to a random walk $\mu$ on the group $ G$  is roughly defined to be {minus} the logarithm of the probability that a random path starting at the first point ever reaches the second \cite{BB}. Its horofunction boundary {$\partial_{\mathcal M}G$} is called the  Martin boundary of $( G, \mu)$ {(see Section 2 for more details)}.

    The geometric boundary we consider is the Floyd boundary. The Floyd metric $\delta^{f}_{o}(.,.)$ at a basepoint $o\in G$ is obtained by rescaling the word metric by a suitable scalar function $f:\Bbb R^+\to \Bbb R^+$. The function $f$ is called  {\it Floyd function} whose definition makes the corresponding Cauchy completion {$\GAF$}   of the Cayley graph to be compact. {The set $\partial_fG=\Gf\setminus G$ is called {the} {\it Floyd boundary}} (see Section 3 for more details).

{One of the main results of the paper  which confirms the above mentioned comparison principle  is  the following inequality which relates the probabilistic metric $d_{\mathcal{G}}$ with the geometric metric $\delta^f_o$.

\begin{theorem}[Theorem \ref{Karlsson}] \label{Ancona}
Let $G$ be a finitely generated group {and $f$ a Floyd function on $G$.} Let $\mu$ be a probability measure on $G$ whose support generates $ G$.  Let $d_{\mathcal{G}}$ be the Green metric associated to $\mu$.

{Assume {that} one of the following {conditions holds}:}

\begin{itemize} \item[a)]{The support of $\mu$ is finite}{; or}

  \item[b)]{The measure $\mu$ has superexponential moment and $x^{2+\alpha}f(x)\to 0$  {$(x\to\infty)$} for some $\alpha>0$. }\end{itemize}

Then there exists a decreasing function $A:\mathbb{R}^{+}\to \mathbb{R}^{+}$ such that {$\foral x,y,z\in G$ one has:}
\begin{equation}\label{relanc} d_{\mathcal{G}}(x,y)+d_{\mathcal{G}}(y,z)\leq d_{\mathcal{G}}(x,z)+A(\delta^{f}_{y}(x,z)).\end{equation}
\end{theorem}

{An analog of the inequality (\ref{relanc}) in the context of   word hyperbolic groups {and finitely supported measures} is due to A.~Ancona \cite{Ancona} and claims that there exists a    constant $C$ such that  one has
\medskip}

  {$\hspace*{2cm}  d_{\mathcal{G}}(x,y)+d_{\mathcal{G}}(y,z)\leq d_{\mathcal{G}}(x,z)+C,\hspace*{3cm} {\sf (Ancona)}$
}
\medskip

\noindent {where the points  $x, y, z$ lie in this order on the geodesic $ G$ in the Cayley graph.}
{The extension to measures of infinite support is due to Gouezel \cite{G2}.}
 {There are two essential differences between the inequality (\ref{relanc}) and the Ancona inequality. Unlike the function $A(\cdot)$,  the constant $C$ in {the Ancona inequality} is a uniform constant (depending on the hyperbolicity constant of the group). On the other hand, in the inequality (\ref{relanc}) the distinct triple $\{x,y,z\}$ does not necessarily belong to one geodesic. }

The Ancona inequality reflects the hyperbolic nature of the metric $d_{{\mathcal G}}$
 in a hyperbolic group. It has sparked a fruitful line of research (see \cite{BHM}, \cite{KaiStrip} for more details).

{We obtain the inequality (\ref{relanc}) as a consequence of the following statement (see Section 5 for a more general statement).}

\medskip

\begin{theorem}
  (Theorem \ref{avoidball}).  {\it There exists a   function $R:\Bbb R^+\times \Bbb R^+\to \Bbb \Bbb R^+$
such that for every $\varepsilon >0$ the probability that a random {path} from $x$ to $y$ passes through a ball centered at $z$
of radius $R(\delta^f_z(x,y))$ is greater than $1-\varepsilon$.}  \end{theorem}

\medskip

{There is another inequality due to A. Karlsson \cite[Lemma 2.1]{Karl} which states the following:}
\medskip

{$\hspace*{3.5cm} d(v, [x, y])\leq K(\delta_{v, f}(x,y))$,\hspace*{2.3cm}{\sf (Karlsson)}}

\medskip

\noindent where $[x,y]$ is a (quasi-)geodesic between the vertices $x$ and $y$ and  $K:\Bbb R^+\to \Bbb N$ is a decreasing function. One can restate
{Karlsson inequality} in the following form:

 \medskip

  {{\sf Karlsson's lemma.} {\it For every $\varepsilon > 0$ there exists $R=R(\varepsilon)$ such that the condition $\delta^f_v(x, y)>\varepsilon$ implies that $d(v, [x, y])\leq R.$ }}

  \medskip

 So if one  replaces the "random path" by "(quasi)-geodesic" the Theorem above becomes Karlsson's lemma and vice versa. The Karlsson inequality in its turn admits many corollaries for relatively hyperbolic groups (see \cite{Ge1}, \cite{Gerasimov}, \cite{GePoGGD}, \cite{GePoCrelle}). It was {one of our initial} motivations to relate the Martin and Floyd compactifications.

 \medskip

  To complete this discussion let us recall the classical Gromov inequality for $\delta$-hyperbolic graphs:

 \medskip

{$\hspace*{2.5cm} d(v , [x, y])-\delta < (x.y)< d(v , [x, y]),\hspace*{2cm} ({\sf Gromov})$}

\medskip

\noindent{where $(x.y)$ is the Gromov product (which is replaced by the Floyd distance in our case)\cite{Gromov}.}

\medskip

 Note that the left-hand side of the Gromov inequality is not true when the Cayley graph is not hyperbolic,  in particular in the case of a relatively hyperbolic groups it is not satisfied for  the horospheres at parabolic points.
{Here an horosphere at a point is the set of all bi-infinite geodesics all based at this point (necessarily not conical). We refer to Section 3 where all standard definitions     are given (e.g. {\it conical points, parabolic points, horospheres} etc).}

We use Theorem \ref{Ancona} {in our next result} to prove that the Martin boundary $\partial_{\mathcal M}  G$ associated to $( G,\mu)$ covers the Floyd boundary $\partial_{f} G$.

\begin{theorem}[Theorem \ref{martintofloyd}]\label{MFmap}
Let $ G$, $\mu$ and $f$ be as in Theorem \ref{Ancona}.
The identity map on $ G$ induces a continuous $ G$-equivariant surjection $\pi: \partial_{\mathcal M}{G}\to  \partial_{f} G$.
Moreover, the preimage of any conical point of {$\partial_{f} G$} is a single point.
\end{theorem}

If every point of $\partial_{f} G$ is conical then the map constructed in Theorem \ref{MFmap} is a homeomorphism, recovering a result of Ancona \cite{Ancona} for hyperbolic groups.

Most of the applications of Theorems \ref{Ancona} and \ref{MFmap} deal with relatively hyperbolic groups.
 If  a group $ G$  is relatively hyperbolic with respect to a finite collection $\mathcal P$ of subgroups  there exists a compactum $T=\partial_{B} G$ ({\it the Bowditch boundary}) on which the action of $G$ is {\it geometrically finite}, i.e. it  is a minimal convergence action and every point of $\partial_{B} G$  is either {\it conical} or {\it bounded parabolic} \cite{Bowditch}.

 {Then  for an exponential Floyd function $f:n\in \Bbb N\to \lambda^n\in \Bbb R\  (\lambda\in (0,1))$, there is
a continuous equivariant surjection \cite{Gerasimov}:}

 $$\phi:\partial_{f} G \to \partial_{B} G.$$
 Moreover, the map $$\psi=\phi\circ \pi:\partial_{\mathcal M}  G\to \partial_{B} G$$
is a continuous $ G$-equivariant surjection with $|\psi^{-1}(q)|=1$ for every conical $q\in \partial_{B} G$ (note, $\partial_{B}X$ contains at most countably many nonconical points).

%{A horosphere $P$ at the parabolic point $p\in \partial_{B} G$ is the set of bi-infinite quasigeodesics based %at $p$ (see Section 3 for more details)}.
{In \cite{DGGP} the authors use Theorem 1.1 as a {crucial} ingredient to precisely identify} {the Martin boundary} {$\partial_{\mathcal M}G$}   {when $G$ is} {relatively} {hyperbolic} {with respect to a system of virtually abelian subgroups.  In particular it is shown in \cite{DGGP}} {that if $p\in \partial_{B}G$ is a} parabolic {point} {with stabilizer} {which contains ${\mathcal Z}^{d}$ {as a subgroup of finite index}{ then $\psi^{-1}(p)$ is homeomorphic to $S^{d-1}$.}

A point $z$ on a (quasi-)geodesic $\alpha$ is called {\it $(\epsilon,R)$-transition point}
if  {for any horosphere $P$ one has}  $\alpha \cap B(v, R)\not\subset
N_\epsilon(P)$  where $B(v,R)$ denotes the  ball centered at $v$ of radius $R$, and $N_{\epsilon}(P)$ is an
$\epsilon$-neighborhood of {$P$.}

 Theorem \ref{Ancona} has the following consequence:

\begin{corollary}[Corollary \ref{relhypAnc2}] \label{relhypAnc}
Let $ G$ be hyperbolic relative to a collection of subgroups, and let $\mu$ be a probability measure on $G$ with superexponential moment and support generating $G$.
{If $x,y,z\in  G$ lie on a word geodesic $\alpha$, and $y$ is  an $(\epsilon,R)$-transition point between $x$ and $z$  then
$$d_{\mathcal{G}}(x,y)+d_{\mathcal{G}}(y,z)\leq d_{\mathcal{G}}(x,z)+A$$ where $A$ depends only on $(\epsilon, R)$,  and $\mu$.}
\end{corollary}

For hyperbolic groups, every point on a word geodesic is a transition point and every point of the Bowditch boundary is conical, {so the above inequality implies the Ancona's inequality for hyperbolic groups.}

We note that {Theorem  \ref{Ancona} and  Corollary} \ref{MFmap} are proved for arbitrary finitely generated groups. { Furthermore, unlike  that of Ancona}, our proofs of them on neither potential theory nor hyperbolic geometry.

{We denote below by $\partial^{min}_{\mathcal M} G$ the set of points of the Martin boundary $\partial_{\mathcal M} G$ which correspond to minimal harmonic functions (see Section  6), every point of this subset is called {\it minimal}. In the following result we describe the subset of minimal points of the preimage of the limit set of a {\it fully quasiconvex} subgroup $H$ of $G$ acting cocompactly on the complementary set of its limit set (see Section 7):}

 \begin{theorem}[Proposition \ref{lsp}]
\label{lsp1}
 {Let $\pi:\pmg\to\pfg$ be the map from Theorem \ref{MFmap}. Let $H<G$ be a subgroup acting cocompactly on $X\setminus\Lambda H$.
Then \begin{equation}\label{Preimbound} \pi^{-1}(\Lambda H)\cap \partial^{min}_{\mathcal M} G\subseteq\psm H,\end{equation} where $\psm H$ denotes the set of accumulation points of $H$ in $\Gm.$}
\end{theorem}

 {As a consequence we obtain that every bounded parabolic subgroup being fully quasiconvex satisfies this Theorem. Furthermore  {it yields} a uniform constant $C$ such that every sequence $(x_n)$ converging to a minimal point in the preimage of a parabolic point $p$ on the Bowditch boundary is situated in a $C$-neighborhood of a sequence $o_n\in H$ (see Corollary \ref{ppoint}}).

\medskip

We use Theorems \ref{Ancona} and \ref{MFmap} to prove some results concerning the harmonic measures on boundaries of {a} group. Consider a group $ G$ acting  by isometries on some proper geodesic Gromov hyperbolic metric space $X$. {We say that the action $G\act X$ is geometrically finite if it is   on $\partial X.$} There are two natural classes of measures on the Gromov boundary $\partial X$ associated with the action. One consists of quasiconformal, or Patterson-Sullivan measures (for lattices in rank 1 symmetric spaces, these coincide with Lebesgue measure).  The other consists of stationary or harmonic measures, which are limits of convolution powers of measures on $ G$. Comparing these two classes {of measures} has been a question of considerable interest {and was our second motivation.} We prove
\begin{theorem}[Theorem \ref{singularityofmeasures}]\label{singularityharQC}
Let $X$ be a proper geodesic Gromov hyperbolic space and $G<Isom(X)$ geometrically finite with at least one parabolic {subgroup.}
 Let $\mu$ be a symmetric probability measure on $ G$ with superexponential moment whose support generates $G$. Let $\nu$ be a {$\mu$-stationary} measure on $\partial X$ and $\kappa$ any $G$ quasiconformal measure on $\partial X$. Then $\nu$ and $\kappa$ are singular.
\end{theorem}
This generalizes a result of Blachere-Haissinsky-Mathieu \cite[Proposition 5.5]{BHM} who proved an analogue where $ G$ is a word hyperbolic group which acts on $X$ with parabolics.
In particular, since Lebesgue measure is conformal for lattices in rank 1 symmetric spaces the following is an immediate corollary:
\begin{corollary}\label{singularityLebesgue}
Let $ G$ be a nonuniform lattice in a rank 1 symmetric space $X$. Let $\mu$ be a symmetric measure on $ G$ with superexponential moment whose support generates $ G$. Then the $\mu$-stationary measure on $\partial X$ is singular to the Lebesgue measure.
\end{corollary}
When $X=\mathcal{H}^{2}$ the analogue of Corollary \ref{singularityLebesgue}  has been independently obtained  by \cite{BHM}, by Gadre, Maher, and Tiozzo in \cite{GMT}, and by Deroin , Kleptsyn and Navas in \cite{DKN}.
Finally,  we construct and study the so called harmonic invariant measure for random walks satisfying the {inequality  (\ref{relanc})}.

\begin{theorem} [{Theorem \ref{invmeasfloyd}}] \label{invmeasurefloyd}
Let $G$ and $\mu$ be as in Theorem \ref{Ancona}, and $\hat{\mu}$ be the reflection of $\mu$. Let $\nu$ (resp. $\hat{\nu}$) be the unique $\mu$ stationary (resp. $\hat{\mu}$ stationary) probability measure on $\partial_{f} G$.
Then there exists a $G${-}invariant Radon measure on $\partial_{f} G\times \partial_{f} G\setminus \Delta(\partial_{f}G)$ in the measure class of $\hat{\nu}_{f} \times \nu_{f}$.
\end{theorem}

{If $X$ is a Riemannian manifold of negative curvature bounded away from $0$, or more generally a proper $CAT(-1)$ space}, {then the following is true.}
\begin{theorem}[Theorem \ref{finiteharmonicinvariant}]\label{tangentbundle}
Let $\mu$ be a finitely supported generating measure on a geometrically finite $ G<Isom(X)$, and $\hat{\mu}$ its reflection. Let $\nu$ (resp. $\hat{\nu}$) be the $\mu$ stationary  (resp. $\hat{\mu}$ stationary) probabiliy measure on $\partial X$. There is a $ G${-}invariant measure $\tilde{L}$ on the unit tangent bundle $T^{1}X=\partial^{2}X \times \mathbb{R}$ in the measure class of $\nu\times \nu \times Leb$ which projects to a finite measure $L$ on $T^{1}X/ G$.
\end{theorem}
We call this the harmonic invariant measure of $\mu$, in analogy with a classical construction where $\mu$ is the Brownian motion. When the action $ G \curvearrowright X$ is convex cocompact such a measure was constructed by Kaimanovich in \cite{Kaierg}
{generalizing results of Anderson and Schoen \cite{AS} for Brownian motion on negatively curved manifolds.}

By the result of \cite{Gequid}, closed geodesics corresponding to loxodromic elements equidistribute with respect to this harmonic invriant measure along typical random walk trajectories.

\section{Random walks on groups}
Let $ G$ be a finitely generated infinite group.
We endow $ G$ with the word distance $d(\cdot,\cdot)$ coming from a finite symmetric generating set $S$.
We set $$||g||=d(e,g)$$
Let $\mu$ be a probability measure on $ G$ whose support generates $G$.
This defines a $ G${-}invariant Markov chain on $ G$ with $n$ step transition probabilities $p_{n}(x,y)=\mu^{*n}(x^{-1}y)$.

{We say $\mu$ has finite support if $$supp (\mu)=\{g\in G: \mu(g)>0\}$$ is a finite set.
We say $\mu$ has exponential (resp. superexponential) moment if
$$\sum_{g\in G}c^{||g||}\mu(g)<\infty$$ for some (resp. for all) $c>1$.}
We define the reflected measure by $\hat{\mu}(g)=\mu(g^{-1})$.
The measure $\mu$ is said to be symmetric if $\hat{\mu}=\mu$.
A trajectory  {$\alpha$} of length $n${, denoted by ${\rm length} (\alpha)$,} is a sequence
{$g_{0},..., g_{n-1}$} of elements of $ G$.
Such a trajectory is said to have jump size bounded by $K$ if $d( g_{i}, g_{i+1})\leq K$ for all $i$.

A trajectory $$\alpha=g_{0},g_{1},...,g_{n}$$ in $ G$ is called $\mu$ admissible if $\mu(g^{-1}_{i}g_{i+1})>0$ for each $i$.
Note, {if $\mu$ has finite support}, an admissible trajectory has jump size bounded by $\displaystyle K=\max_{g\in supp(\mu)}||g||$.

Given an admissible trajectory, its weight is defined to be $$w(\alpha)=\mu(g^{-1}_{0}g_{1})\mu(g^{-1}_{1}g_{2})...\mu(g^{-1}_{n-1}g_{n}).$$

Let $Traj(x,y)$ denote the set of all admissible trajectories in $ G$ which begin at $x$ and end at $y$.
Let $Traj_{r}(x,y)\subset Traj(x,y)$ consist of trajectories of length $r$.
The Green's function associated to $\mu$ is defined as
$$\mathcal{G}(x,y)=\sum_{\alpha \in Traj(x,y)}w(\alpha).$$
{The $\mu$ random walk is called {\it transient} if  the probability of ever returning to the start point is less than $1$.
In this case, $\mathcal{G}(x,y)<\infty$ for all $x,y\in  G$ {; in the opposite case the random walk is called {\it recurrent}} \cite{Woess}.
If $G$ contains $\mathbb{Z}$ as a finite index subgroup, any measure  on $G$ induces a recurrent random walk. The same is true when $G$  contains $\mathbb{Z}^{2}$
as a finite index subgroup and $\mu$ has exponential moment.
Conversely, by work of Varopoulos \cite[Theorem 4.6]{Varopoulos}, if there is a measure $\mu$ on $G$ whose support generates $G$ and the $\mu${-}random walk is recurrent, then $G$ is either finite or contains $\mathbb{Z}$ or $\mathbb{Z}^{2}$ {as a finite index subgroup}.
We will from now on assume {that} the $\mu${-}ran\-dom walk is transient.}
Note $$p_{n}(x,y)=\sum_{\alpha \in Traj_{r}(x,y)}w(\alpha).$$

For each $x,y\in  G$ one can define a probability measure $P_{x,y}$ on
the set $Traj(x,y)$ of trajectories from $x$ to $y$ as follows:
for $V\subset Traj(x,y)$

\begin{equation}\label{probset} P_{x,y}(V)=\frac{1}{\mathcal{G}(x,y)}\sum_{\alpha \in V} w(\alpha)\end{equation}

For a subset $V\subset  G$ let $V^{c}$ denote the complement of $V$ in $ G$.
For $V\subset  G$ let

$\mathcal{G}(x,y,V)$ be the total weight of trajectories from $x$ to $y$ which are contained in $V$, except possibly for the endpoints.

For a real number $r$ define
\begin{equation}\label{weightedgreen}
\mathcal{G}(x,y|r)=\sum^{\infty}_{n=0}r^{n}p^{n}(x,y)
\end{equation}
It is easy to see that $\mathcal{G}(.,.|r)$ is $G$ equivariant, i.e.
$$\mathcal{G}(gx,gy|r)=\mathcal{G}(x,y|r)$$ for all $x,y,g\in G$, $r>0$.

When the support of $\mu$ generates $G$ as a semigroup, the convergence of the series in (\ref{weightedgreen}) does not depend on $x,y$ (see e.g. \cite[Lemma 1.7]{Woess}). Consequently, the radius $r(\mu)$ of convergence of $G(x,y|.)$ is independent of $x,y\in G$.

Note, 
$$r(\mu)=\lim \inf_{n\to \infty}p^{n}(x,y)^{-1/n}.$$

The number $\rho(\mu)=1/r(\mu)$ is called the spectral radius of $\mu$.
Kesten \cite{Kesten1}, \cite{Kesten2} and Day \cite{Day} proved that $\rho(\mu)<1$ whenever $G$ is nonamenable and the support of $\mu$ generates $G$ as a semigroup.

The following is the Harnack inequality, valid for any full-support random walk on a finitely generated group:
\begin{lemma}\label{Harnack}
For each $t\in (0, r(\mu))$ there is a {$\lambda=\lambda_{t}\in (0,1)$} such that $\mathcal{G}(x,y|t)\geq \mathcal{G}(x,z|t)\lambda^{d(y,z)}$ for all $x,y,z\in  G$
\end{lemma}

This easily implies:
\begin{corollary}\label{harnackcor}
For each $t\in (0, r(\mu))$ there is an {$L_{t}>1$} such that $$L^{-d(x,y)}_{t} \leq
\mathcal{G}(x,y|t)\leq L^{d(x,y)}_{t}$$ for all $x,y,z\in  G$
\end{corollary}

%Letting $x=z=e$ we obtain
%$\mathcal{G}(e,y|t)\geq \mathcal{G}(e,e|t)\lambda^{d(y,e)}$ and by $G$ equivariance of $\mathcal{G}$ we get:
%\begin{equation}\label{decaybound}
%\mathcal{G}(x,y|t)\geq \mathcal{G}(e,e|t)\lambda^{d(x,y)}
%\end{equation}

We will need the following. \begin{proposition}\label{Probabilistic Lemma}
If $ G$ is nonamenable and the support of $\mu$ generates $G$ as a semigroup,  there exists $0<\phi<1$ and $D>0$ such that for any $x,y\in  G$ {and $M\in \Bbb N$ one has}
\begin{equation}\label{phibound}P_{x,y}( G \in Traj(x,y): length( G))\geq {M})\leq \phi^{{M}-D d(x,y)}. \end{equation}
\end{proposition}

\begin{proof}
Since $\Gamma$ is nonamenable, $r(\mu)>1$.
Let $t\in (1,r(\mu)).$
Then
$$G(x,y|t)=\sum^{\infty}_{n=0}t^{n}p^{n}(x,y)$$ converges for all
$x,y\in G$.
Let $\phi=1/t$ and $L=\max(L_{1},L_{t}).$

We have

$$\sum_{n\geq M}p^{n}(x,y)\leq t^{-M}\sum_{n\geq M}t^{n}p^{n}(x,y)
\leq t^{-M}\mathcal{G}(x,y|t)\leq \phi^{M}L^{d(x,y)}.$$

On the other hand,
$$\mathcal{G}(x,y)\geq L^{-d(x,y)}.$$

Thus we obtain
$$\sum_{n\geq M}p^{n}(x,y)\leq
\phi^{M}L^{d(x,y)}\leq \phi^{M}L^{2d(x,y)}\mathcal{G}(x,y)=\phi^{M-Dd(x,y)}\mathcal{G}(x,y)$$ where $D=2 \log_{t}L>0$.
\end{proof}

\section{Background on {convergence groups and   Floyd compactification{s}}}
By a \it graph \rm we mean a pair $(\Delta^0,\Delta^1)$ where
$\Delta^0$ is a set and $\Delta^1$ is a set of subsets
of cardinality 2 of $\Delta^0$.

%Let $\Delta$ be a graph. Denote by $\Delta^0$ and $\Delta^1$ the set

%of vertices and the set of edges of $\Delta$ respectively.

A \it path \rm in $\Delta$ is a map $J\overset \gamma\to\Delta^0$
where $J$ is a finite nonempty convex subset of $\Bbb Z$, such that
$\{ \gamma(i), \gamma(i{+}1)\}{\in}\Delta^1$ for all $i{\in}J{\setminus}\{\s{max}J\}$.
The \it length \rm of such a path $\gamma$ is the number $\s{max}J{-}\s{min}J$.

If $\s{min}J{=}a,\s{max}J{=}b$ we write $J{=}\overline{a,b}$.

For $x,y$ let $\s{Path}_\Delta(x,y)=
\s{Path}(x,y)=\{ \gamma: \gamma$ is a path
$\overline{0,n}\to\Delta^0$ for some nonnegative integer
$n$ such that $ \gamma(0){=}x, \gamma(n){=}y\}$.

%A graph $\Delta$ is \it connected \rm if $\s{Pth}(x,y){\ne}\varnothing$
%for every $x,y{\in}\Delta^0$.

Suppose that $\Delta$ is connected.
So the ``standard'' distance function $d$ on $\Delta^0$ is given by\hfil\penalty-10000
$d_\Delta(x,y)=
d(x,y)=\s{min}\{\s{length}(\gamma): \gamma{\in}\s{Path}(x,y)\}$.

Let $\Bbb R_{>0}\overset f\to\Bbb R_{>0}$ be a nonincreasing function.
We use $f$ for rescaling the distance $d$ as follows.
Let $v{\in}\Delta^0$ be a ``basepoint''.
For $e{\in}\Delta^1$ we declare that the $(f,v)$\it-length \rm of the edge
$e$ is equal to $f(d(e,v))$.
The $(f,v)$\it-length \rm of a \bf path \rm$J\overset \gamma\to\Delta^0$
is the number\hfil\penalty-10000
$\s{length}_v^f(\gamma)=\sum_{j\in J\setminus\{\s{max}J\}}\s{length}_v^f\{ \gamma(j), \gamma(j{+}1)\}$,
and the $(f,v)$\it-distance \rm function\hfil\penalty-10000
$\delta^f_v(x,y)=\s{min}\{\s{length}_v^f \gamma: \gamma{\in}\s{Path}(x,y)\}$
is well-defined.

We suppose that the graph $\Delta$ is \it locally finite\rm, i.e,
the set of edges containing each vertex $v{\in}\Delta^0$ is finite.

If the ``rescaling function'' $f$ satisfies the condition\hfil\penalty-10000
{\begin{equation}\label{series} \sum_{k=0}^\infty f(k)<\infty\end{equation}}
 then the Cauchy completion
of the metric space $(\Delta^0,\delta^f_v)$ is compact.

Now we impose on $f$ one more condition:

{\begin{equation}\label{coef} \exist\kappa\geq 1\ \foral n\in\Bbb N\ :\   \Frac{f(n)}{f(n{+}1)}\leq \kappa.\end{equation}}

Any nonincreasing function satisfying {(\ref{series}) and (\ref{coef})} is called a
\it Floyd rescaling function\rm.
For such a function, the Cauchy completion ${\overline{\Delta}_f}$
(called the \it Floyd compactification \rm of $\Delta$ with respect to $f$)
does not depend on the choice of the base point $v$ and
every isometry of the metric space $(\Delta,d)$ is uniformely
continuous with respect to $\delta_f$ and hence extends to a homeomorphism
${\overline{\Delta}_f}\to{\overline{\Delta}_f}$.

The distance function $\delta^f_v$ in the case when $f(x){=}\Frac1{x^2{+}1}$
was introduced by W. Floyd in \cite{Floyd} who used it to study limit sets of
geometically finite Kleinian group. So we will call $\delta^f_v$ the
\it Floyd distance\rm.

The {complement} $\partial_f\Delta={\overline{\Delta}_f}\setminus \Delta^0$
is the \it Floyd boundary \rm of $\Delta$ with respect to $f$.

Suppose that $\Delta$ is a Cayley graph of a group $G$ with respect to
a finite generating set $\Cal S$. {We denote by $d(\cdot, \cdot)$ and $||\cdot||$ the word distance and its norm in the graph $(\Delta, \Cal S)$. For a fixed system $S$ rescaling the distance $d(v, \rm{edge})$ by a function $f$ we obtain in the same way  the Floyd compactification of $G$ and its boundary  denoted respectively by
 $\Gf$ and $\partial_f G.$}

{\it Remark{s}.} If $\underset{x\rightarrow\infty}{\s{lim\,sup}}
 \Frac{f(x)}{f(x{+}1)}{=}1$
 then the Floyd compactification does not depend on the choice
 of finite generating set {(see \cite[Lemma 2.5]{GePoJEMS} or \cite[Corollary 7.7]{GePoGGD}} for more details).

%So the notion of the \it Floyd compactification \rm and the
%\it Floyd boundary \rm of a group with respect to a Floyd scaling function
%is well defined in this case.

 For the construction of the Floyd compactification to make sense it suffices
to consider a function $\Bbb N\overset f\to\Bbb R_{>0}$ defined
on the set of positive integers. However we {extend the definition of $f$ to the set  of positive real numbers $\Bbb R_{>0}$  which will  simplify} some
calculations (see the formulas of Section 4 below).

\medskip

 {For the reader's convenience we recall now few standard definitions currently used in the paper.
An action of $G$ on a compactum $T$ is {\it convergence} if the induced action on the set of distinct triples of $T$ is discontinuous.}
{Suppose $G\curvearrowright T$ is a convergence action. The set of accumulation points  $\Lambda G$ of any orbit $Gx\ (x\in T)$ is called {\it limit set} of the action. As long as $\Lambda G$ has more than two points, it is uncountable and the unique minimal closed $G$-invariant subset of $T.$.
The action is  then said to be nonelementary. In this case, the orbit of every point in $\Lambda$ is infinite.
If $G$ admits a nonelementary convergence action, then $G$ must contain a free subgroup of rank 2, hence in particular is non-amenable.}

The action $G\curvearrowright T$ is {\it minimal} if $\Lambda G=T.$ There is a natural topology on the disjoint union ({\it attractor sum}) $X=G\sqcup \Lambda G$ such that the action $G\act X$ is also convergence \cite{Gerasimov}. In particular if $T=\partial_fG$ is the Floyd boundary then the action on $T$ is convergence \cite{Karl} and so is on $\Gf=G\sqcup \partial G_f.$}

{A point $\zeta\in\Lambda G$  is called  {\it conical} if there is a   sequence $g_{n}\in  G$ and distinct points $\alpha,\beta \in \Lambda G$ such that
$g_{n}\zeta \to \alpha$ and $g_{n}\eta \to \beta$ for all $\eta \in  X\ \setminus\{\zeta\}.$}

{A point $p\in \Lambda G$ is {\it bounded parabolic} if it is the unique  (parabolic) fixed point of its stabilizer ({\it maximal parabolic}) subgroup  $H$, which acts cocompactly on $\Lambda\setminus\{p\}$.
B.~Bowditch   proved that if $G$ is a relatively hyperbolic group then there exists a compactum $T$ on which the action is minimal, convergence and the action $G\act (T=\Lambda G)$ is geometrically finite, i.e. every point of $T$ is either conical or bounded parabolic \cite{Bowditch}. Furthermore  the action of $G$ extends to a convergence action on the compactum  $\overline{G}_{\mathcal B}=G\sqcup \Lambda G$ which we call {\it Bowditch compactification}.   In its turn the existence of a geometrically finite action of a finitely generated group on a metrizable compactum implies that the group is relatively hyperbolic with respect to the system of the parabolic points stabilizers \cite{Ya}. So the existence of a geometrically finite action can be taken as a definition of the relative hyperbolicity (the proof that this is equivalent to several other dynamical definitions can be found in \cite{Ge1}, \cite{Gerasimov}, \cite{GePoGGD})}

{A bi-infinite quasigeodesic  $\gamma :
\Bbb Z \to G$ is a {\it horocycle}  at $p\in
T$ if $\displaystyle\lim_{n\to\pm\infty}\gamma(n){=}p$.  The unique limit point $p$ of $\gamma$  is not conical \cite[Proposition 4.4.1]{GePoCrelle} and is called {\it base}
of the horocycle. A {\it horosphere} $P$ at the parabolic point $p$ is the set of all horocycles based at $p.$ On can equivalently define the horosphere as a  neighborhood of a left coset $gH\ (g\in G)$ where $H$ is the stabilizer of the parabolic point $p$ (see \cite{GePoCrelle}).}

{The Floyd compactification has been instrumental in studying  {\it relatively hyperbolic groups.}
{Indeed, whenever $G\curvearrowright B$ is a nonelementary geometrically finite minimal action on a compactum, Gerasimov proved
that there exists a positive $\lambda{\in}(0,1) $ such that for every function $f:\R_{>0}\to\R_{>0}$ satisfying the conditions (\ref{series}), (\ref{coef})  and $f(x)\leq \lambda^x\ (x\in\R)$ there exists a continuous equivariant surjection
$F:\partial_{f}G\to B$}  \cite[Proposition 3.4.6]{Gerasimov}.}

\section{Proof of Theorem 1.1{{\rm .a}}: Geometric part}

{Let} $G$ be a finitely generated group {equipped} with a word distance $d(\cdot, \cdot)$.

For a basepoint $o\in G$ let $\delta^{f}_{o}(x,y)$ denote the Floyd distance based at $o$ with respect to the rescaling function $f$.

Given a symmetric measure $\mu$ on $G$ Blachere and Brofferio \cite{BB} introduced a metric {$d_{\mathcal{G}}$} on $G$, called the Green metric, {given} by
$$d_{\mathcal{G}}(x,y)=-{\ln}\frac{\mathcal{G}(x,y)}{\mathcal{G}(e,e)}.$$
When $\mu$ is not symmetric, $d_{\mathcal{G}}$ still defines an asymmetric metric on $G$.
The expression makes sense whenever the Markov chain defined by $\mu$ is transient.
By \cite[Lemma 3.6]{BHM}, $d_{\mathcal G}$ is quasi-isometric to the word metric whenever $G$ is nonamenable and $\mu$ is symmetric and has exponential moment.
The goal of the next two sections is to prove the following, {assuming $\mu$ has finite support generating $G$.}

\begin{theorem}\label{Karlsson}  {\sf(Ancona-Karlsson type inequality). There is a function $A:\R^{+}\to \R^{+}$ such that for all $x,y,z\in G$ one has
\begin{equation}\label{AncKarl} d_{\mathcal{G}}(x,y)\geq d_{\mathcal{G}}(x,z)+d_{\mathcal{G}}(z,y)-A(\delta^{f}_{z}(x,y)).\end{equation}}
\end{theorem}

\proof
{By equivariance we can assume $z=o$ is a fixed basepoint.
Note, if $G$ is amenable, Karlsson showed that $|\partial_{f}G|\leq 2$ \cite[Corollary {2}]{Karl}.
%If the Floyd boundary is empty thenso for any $\delta>0$ and $x\in G$ there are only finitely many $y$ with $\delta^{f}_{o}(x,y)>\delta$. Thus (\label{AncKarl}) holds automatically.
If the Floyd boundary is empty, $G$ is finite so the $\mu$ random walk is recurrent. The same is true when $|\partial_{f}G|=2$, since in that case $G$ is virtually $\mathbb{Z}$.
We now treat the case $|\partial_{f}G|=1$.
Then for each $\delta>0$ there is an $R$ such that if $x,y$ are both outside  {the ball} $B_{R}o$ (in the word metric) then $\delta^{f}_{o}(x,y)<\delta$.
So, if $\delta^{f}_{o}(x,y)>\delta$ we know that one of $x,y$ (say $x$) is in $B_{R}o$. By the Harnack inequality, this means $C^{-1} \leq G(x,o)\leq C$ and $C^{-1}\leq G(x,y)/G(o,y)\leq C$, where $C>1$ depends only on $R$ (hence only on $\delta$) so we must have
$G(x,y)\leq C^{2}G(x,o)G(o,y)$ whenever $\delta^{f}_{o}(x,y)>\delta$ and taking logarithms we obtain the desired inequality (\ref{AncKarl}) with  $A(\delta)={-2\ln C}$.
From now on we assume $G$ is nonamenable.}

Fix a constant $\tau>1$ (it will suffice throughout to consider $\tau=2$).
We begin with an elementary lemma.
\begin{lemma}\label{helpfunction}
There exists a function $e:\mathbb{R}^{+}\to \mathbb{R}^{+}$  such that as  $r\to\infty$ we have:

%\begin{enumerate}

 %\item[]

 \begin{equation}  \label{decre}
e(r)\to 0,\end{equation}

and

%\item[]
\begin{equation}\label{edecre}\frac{e(r)-e(\tau r)}{rf(r)}\to \infty.\end{equation}
%\end{enumerate}
\end{lemma}
\begin{proof}

Let
$$\alpha(s)=\int^{\infty}_{s}f(t)dt$$ and
$$g(t)=\frac{f(t/\tau)}{\alpha(t/\tau)^{1/2}},$$

{The function $\alpha(\cdot)$ is well-defined by the condition (\ref{series}) of the last section, and $\displaystyle \lim_{s\to\infty}\alpha(s)=0.$}

We also claim that the integral $\int^{\infty}_{0}g(t)dt$ converges.
Indeed,  we have
$$\frac{d\alpha}{ds}=-f(s)$$
Thus for every $M>0$ we obtain

$$\int^{M}_{0}\frac{f(t/\tau)}{\alpha(t/\tau)^{1/2}}dt=\tau \int^{\alpha(0)}_{\alpha(M/\tau)} \frac{d\alpha}{\sqrt\alpha}=2\tau\Big(\sqrt{\alpha(0)}-\sqrt{\alpha(M/\tau)}\Big)\leq 2\tau  \sqrt{\alpha(0)}.$$

Therefore the  function
\begin{equation}\label{er}
e(r)=\int^{\infty}_{r}g(t)dt
\end{equation}
is {also} well defined {and the condition (\ref{decre}) is satisfied.}

On the other hand, by the mean value theorem there is an $s\in [r,\tau r]$
with $$e(r)-e(\tau r)=(\tau r-r) g(s).$$
Thus $$\frac{e(r)-e(\tau r)}{rf(r)}=\frac{(\tau-1)g(s)}{f(r)}=(\tau-1)\frac{f(s/\tau)}{f(r)\alpha(s/\tau)^{1/2}}\geq \frac{\tau-1}{\alpha(s/\tau)^{1/2}}$$ $$\geq \frac{\tau-1}{\alpha(r/\tau)^{1/2}}\to \infty.$$
\end{proof}
\begin{remark}
If $f$ satisfies $f(r)\leq r^{-1-\epsilon}$ for some $\epsilon>0$  we can use {in the argument above} the simpler expression $e(r)=\frac{1}{\log(r)}.$
\end{remark}

For $S\subset G$ we denote by $N_{r}S\subset G$ the $r$-neigborhood of $S$ with respect to the word metric {$d$}.
Denote also by $N^{f}_{r}S$  the $r$-neighborhood of $S$ in the Floyd metric $\delta^{f}_{o}$.

Let $e:\R\to R$ be a function satisfying Lemma \ref{helpfunction}.

Let $E_{r}(x)=N_{r}o\cap N^{f}_{e(r)}x.$
%{For a constant $K>0$ a path $\gamma:\Bbb Z\to G$ is said $K$-{\it admissible} (or simply admissible if the choice of  $K$ is not important). }

The following geometric estimate is crucial for the proof of Theorem \ref{relancona}.
\begin{proposition}\label{Geometric Proposition}For any $K>0$ there are functions $R_{0}:\R\to \R$ and $h:\R \to \R$ with $h(r)/r\to \infty$ as $r\to \infty$
such that for all $x,y \in G$ and all $r>R_{0}(\delta^{f}(x,y))$, for each $u\in E_{\tau r}(x)$ and
$v\in E_{\tau r}(y)$, any trajectory  from $u$ to $v$ disjoint from $E_{r}(x)$ {with jump size bounded by $K$} has length at least $h(r)$.
\end{proposition}

\begin{proof}
See the figure below for an illustration.
Denote by $\delta$ the Floyd distance $\delta^{f}_{o}(x,y)$.
Let $\gamma=\gamma_{0},...,\gamma_{N-1}$ be a trajectory of length $N$, with jump size bounded by $K$, from $u=\gamma_{0}\in E_{\tau r}(x)$  to $v=\gamma_{N-1}\in E_{\tau r}(y)$ and not intersecting $E_{r}(x)$. Let $length(\gamma)=N$ and $$l^f_o(\gamma)=\sum^{N}_{n=1}\delta{_o^f}(\gamma_{n},\gamma_{n-1}).$$

First consider the case when $\gamma$ does not pass through $N_{r}o$. Then for $0\leq n <N$ any unit speed trajectory between $\gamma_n$ and $\gamma_{n+1}$ does not pass through $N_{r-K/2}o$,
so $$\delta^{f}_{o}(\gamma_{n},\gamma_{n+1})\leq d(\gamma_{n},\gamma_{n+1})f(r-K/2)\leq K f(r-K/2)$$
for each $0\leq n \leq N$.
So $$\delta^f_o(u,v)\leq l^{f}_{o}(\gamma)\leq
  K \cdot length (\gamma)\cdot f(r-K/2).$$

  We obtain $$length(\gamma)\geq \frac{\delta^{f}_{o}(u,v)}{K f(r-K/2)}\geq \frac{\delta-2e(\tau r)}{K f(r-K/2)}.$$

On the other hand, if $\gamma$ does pass through $N_{r}o$ let $\gamma(i_{0})$ and $\gamma(i_{1})$  be the first and last intersection of $\gamma$ with $N_{r}o$ respectively.
Since $$d(\gamma_{n},\gamma_{n+1})\leq K\  {(n\in\Bbb N)}$$ {we still have}
$$\delta^{f}_{o}(\gamma_{n},\gamma_{n+1})\leq f(r-K/2)d(\gamma_{n},\gamma_{n+1})\ {\rm for\ all}\ 0\leq n< i_{0}.$$
{The trajectory} $\gamma$ does not intersect $E_{r}(x)$ {so
$\delta^f_o(x,\gamma(i_0))\geq e(r).$}
Since $u\in E_{\tau r}(x)$ we have $\delta^f_o(x,u)\leq e(\tau r).$  {It follows} {that}
$$length(\gamma)\cdot f(r-K/2)\geq length(\gamma\vert_{[0, i_0]})\cdot f(r-K/2)$$

 $$\geq {1\over K}\sum_{0\leq n<{i_{0}-1}}f(r-K/2)d(\gamma_{n},\gamma_{n+1})\geq {1\over K}  \sum_{0\leq n<{i_{0}-1}}\delta^{f}_{o}(\gamma_{n},\gamma_{n+1})$$

$$\geq {1\over K}\delta_o^f(\gamma(i_0{)}, u) \geq  {1\over K}\vert\delta^f_o(x,\gamma_{i_0})-\delta^f_o(x,u)\vert\geq {1\over K}(e(r)-e(\tau r)),$$

\noindent where $\gamma\vert_{[0, i_0]}$ denotes the restriction of $\gamma$ to $[0, i_0].$

Since the function (\ref{er}) decays to zero there exists $R_0=R_{0}(\delta)$ such that for all $r\geq R_0$ we have \begin{equation}\label{lbound} \delta \geq e_{r}+e_{\tau r.}\end{equation} {By (\ref{coef}) there exists a constant $\kappa > 1$ such that  $\displaystyle f(r-K/2){\leq} \kappa^{[K/2+1]}f(r)$ where $[\cdot]$ denotes the integer part of a number.} Therefore in both cases we obtain
$$length(\gamma)\geq \frac{e(r)-e(\tau r)}{Kf(r-K/2)}\geq C(K)\frac{e(r)-e(\tau r)}{Kf(r)}$$
where {$\displaystyle C(K)={\kappa^{[K/2+1]} \over K}.$}
 {Set $$h(r)=C(K)\frac{e(r)-e(\tau r)}{Kf(r)}.$$}  {It follows from (\ref{edecre})  that} $h(r)/r\to \infty$ as $r\to \infty$ completing the proof.

\end{proof}
\noindent\begin{picture}(0,230)(0,0)
\put(-140,-500){\includegraphics*{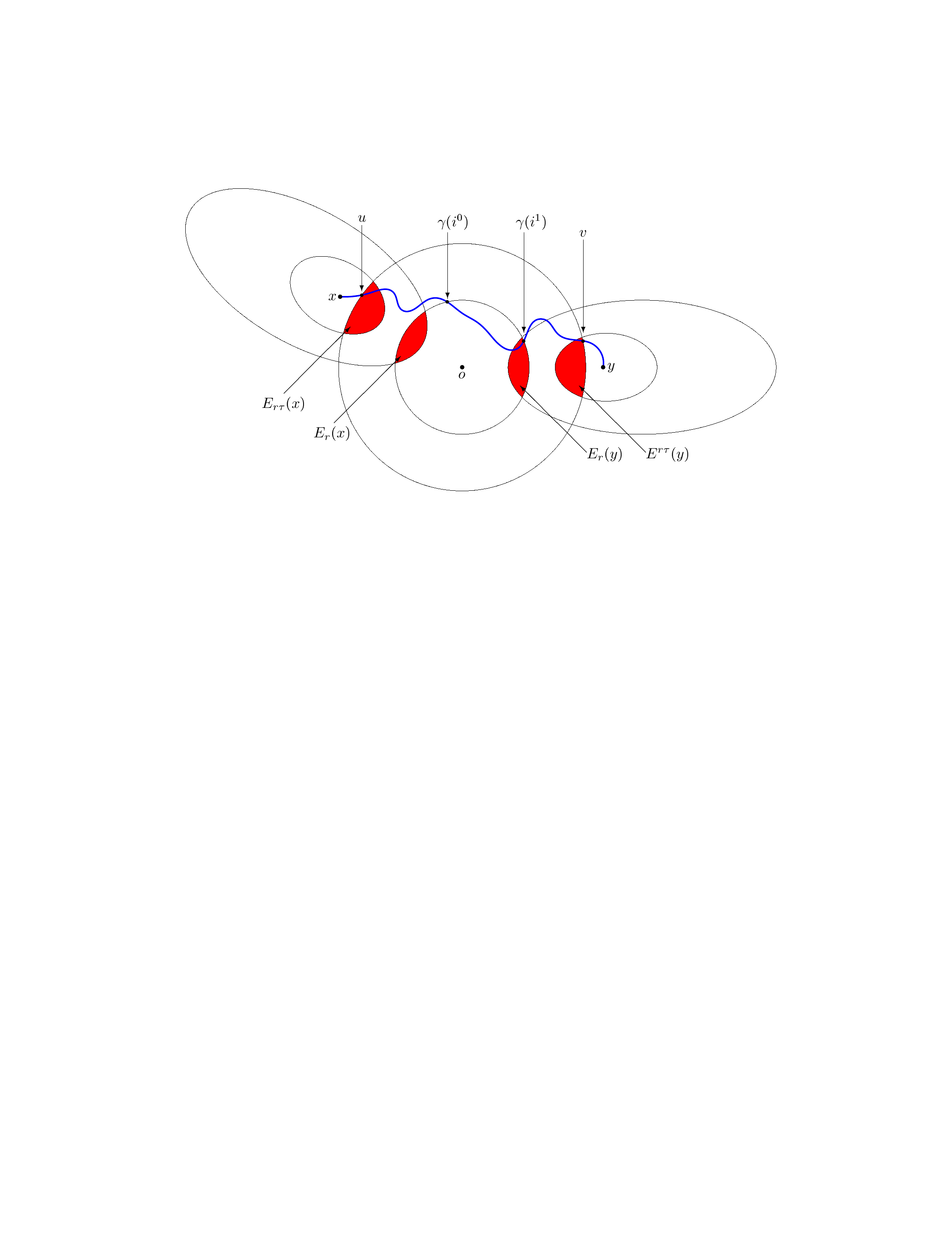}}
\end{picture}

\section{End of {the proof} of Theorem 1.1{.{\rm a}}: Probabilistic part}
The goal of this section is to prove the following multiplicative version of Theorem \ref{Karlsson}  \begin{theorem}\label{relancona}
There is a decreasing function $S:\R^{+}\to \R^{+}$ such that for all $w,x,y\in G$
\begin{equation}\label{anconamult}\mathcal{G}(x,y)\leq S(\delta^{f}_{w}(x,y))\mathcal{G}(x,w)\mathcal{G}(w,y)\end{equation}
\end{theorem}
As noted before, we can assume $G$ is nonamenable.
We then obtain (\ref{anconamult}) {as a corollary of} the following:
\begin{theorem}\label{avoidball}
{There is a {decreasing} function $R:\mathbb{R}^{+}\times \mathbb{R}^{+} \to \mathbb{R}^{+}$ such that for each $\varepsilon >0$ and for all $x,y,w\in G$ one has}
\begin{equation}\label{karlssonprob}P_{x,y}(\gamma \in Traj(x,y):\gamma\cap {N_{R(\epsilon,\hspace*{0.5mm}\delta^{f}_{w}(x,y))}w\neq \emptyset)}>1-\epsilon,\end{equation} {where $N_{R(\epsilon,\hspace*{0.5mm} \delta^{f}_{w}(x,y))}w$ is the ball centered at $w$ of radius $R=R(\epsilon,\hspace*{0.5mm}\delta^{f}_{w}(x,y))$}.
\end{theorem}
\begin{proof}[Proof of Theorem \ref{relancona} from Theorem \ref{avoidball}]
Let $\epsilon=1/2$ and $R(t)=R_{1/2}(t)$ given by Theorem \ref{avoidball}.
Then
$$P_{x,y}(\gamma \in Traj(x,y):\gamma\cap N_{R_{\epsilon}(\delta^{f}_{w}(x,y))}w\neq \emptyset)>1/2.$$
This implies, with
$R=R_{{1/2}}(\delta^{f}_{w}(x,y))$
$$\mathcal{G}(x,y)\leq 2\sum_{z\in N_{R}w}\mathcal{G}(x,z)\mathcal{G}(z,y)$$

By the Harnack inequalities there is a {constant $L$}, depending only on $(G, \mu)$ with
$${L^{-d(z,w)}\leq \mathcal{G}(p,z)/\mathcal{G}(p,w)\leq L^{d(z,w)}}$$ for all $p\in G$ and $z\in N_{R\hspace*{0.3mm}}w$.
Thus,
$$\mathcal{G}(x,y)\leq 2L^{2{R}}|N_{R\hspace*{0.4mm}}{w}|\mathcal{G}(x,w)\mathcal{G}(w,y)$$ for $R=R(\delta^{f}_{w}(x,y)).$ \end{proof}
The rest of this section is devoted to proving Theorem \ref{avoidball}.

First, note by equivariance we can assume {that} $w$ is our fixed basepoint $o\in G$.
{L}et $Q_{r}(x,y)$ be the set of trajectories from $x$ to $y$ which pass both $E_{\tau r}(x)$ and $E_{\tau r}(y)$ but either do not pass $E_{r}(x)$ or do not pass $E_{r}(y)$.
We will use Proposition \ref{Geometric Proposition} together with Proposition \ref{Probabilistic Lemma} to prove the following.
\begin{lemma}\label{longpathsunlikely}
For all $x,y \in G$ and all $r>R_{0}(\delta^{f}_{o}(x,y))$ we have
\begin{equation}\label{probubound}P_{x,y}(Q_{r}(x,y){\bf )}\leq \phi^{h(r)-2D\tau r},\end{equation} where the functions $h$ and $R_{0}$ come from Proposition \ref{Geometric Proposition}{, and the constants $D$ and $\phi$ from Proposition \ref{Probabilistic Lemma}}.
\end{lemma}

\begin{proof}
First note that any trajectory $\gamma$ with $P_{x,y}(\gamma)>0$ has jump size bounded by $\displaystyle K=\max_{g\in supp(\mu)}||g||$.

Since a trajectory in $Q_{r}(x,y)$ misses either $E_{r}(x)$ or $E_{r}(y) $ we have
$$P_{x,y}(Q_{r}(x,y))\leq$$ $$P(\gamma \in Traj(x,y):\gamma \cap E_{r}(x)=\emptyset, \gamma \cap E_{\tau r}(x) \neq \emptyset, \gamma \cap E_{\tau r}(x) \neq \emptyset)+$$ $$P(\gamma \in Traj(x,y):\gamma \cap E_{r}(y)=\emptyset, \gamma \cap E_{\tau r}(x) \neq \emptyset, \gamma \cap E_{\tau r}(y) \neq \emptyset)$$

Let us estimate the first of these (by symmetry, the second is the same).

The total weight of trajectories from $x$ to $y$ which pass both $E_{\tau r}(x)$ and $E_{\tau r}(y)$ but not $E_{r}(x)$ is:

\begin{equation}\label{outer} \mathcal{G}(x,y)P_{x,y}(\gamma \in Traj(x,y):\gamma \cap E_{r}(x)=\emptyset, \gamma \cap E_{\tau r}(x) \neq \emptyset, \gamma \cap E_{\tau r}(y) \neq \emptyset)=$$ $$\sum_{u\in E_{\tau r}(x), v\in E_{\tau r}(y)}\mathcal{G}(x,u,E^{c}_{\tau r}(x))\mathcal{G}(u,v,E^{c}_{r}(x))\mathcal{G}(v,y,E^{c}_{\tau r}(y)){.}\end{equation}

Here $u$ is the first entrance point of $\gamma$ into $E_{\tau r}(x)$, $v$ is the last exit point
out of $E_{\tau r}(y)${, and   $E^c$ denotes the complement of a set $E$}.

On the other hand the total weight of all trajectories from $x$ to $y$ which pass both $E_{\tau r}(x)$ and $E_{\tau r}(y)$ (possibly also passing $E_{r}(x))$ is:

\begin{equation}\label{totalpr} \mathcal{G}(x,y)P_{x,y}(\gamma \in Traj(x,y):  \gamma \cap E_{\tau r}(x) \neq \emptyset, \gamma \cap E_{\tau r}(y) \neq \emptyset)=$$
$$\sum_{u\in E_{\tau r}(x), v\in E_{\tau r}(y)}\mathcal{G}(x,u,E^{c}_{\tau r}(x))\mathcal{G}(u,v)\mathcal{G}(v,y,E^c_{\tau r}(y))\end{equation}
Note the only difference {between  (\ref{totalpr})} and (\ref{outer}) is in the middle factor.

%\st{***(Explanation of the expressions in words: any path from $x$ to $y$ that passes through both $E_{\tau r}(x)$ and $E_{\tau r}(y)$ has a unique first entrance point $u$ into $E_{\tau r}(x)$ and a unique last exit point $v$ out of $E_{\tau r}(y)$. Thus, the path from $x$ to $y$ can be decomposed into a subpath from $x$ to $u$ disjoint from $E_{\tau r}(x)$, a subpath from $u$ to $v$ and a subpath from $v$ to $y$ disjoint from $E_{\tau r}(y)$.If the original path from $x$ to $y$ is disjoint from $E_{r}(x)$, so is the subpath from $u$ to $v$. Thus, for each $u\in E_{\tau r}(x)$ and $v\in E_{\tau r}(y)$ the total weight of paths from $x$ to $y$ which first enter $E_{\tau r}(x)$ at $u$ and last exit $E_{\tau r}(y)$ at $v$ while avoiding $E_{r}(x)$ is $\mathcal{G}(x,u,E^{c}_{\tau r}(x))\mathcal{G}(u,v,E^{c}_{r}(x))\mathcal{G}(v,y,E^{c}_{\tau r}(y))$.

%On the other hand, the total weight of all paths from $x$ to $y$ which first enter $E_{\tau r}(x)$ at $u$ and last exit $E_{\tau r}(y)$ at $v$ is $\mathcal{G}(x,u,E^{c}_{\tau r}(x))\mathcal{G}(u,v)\mathcal{G}(v,y,E^{c}_{\tau r}(y))$).*** {\footnotesize Up to me this last remark is not needed as everything is explained before.}

{By}  Proposition \ref{Geometric Proposition}, if $r>R_{0}(\delta^{f}(x,y))$, any trajectory from $u\in E_{\tau r}(x)$ to $v \in E_{\tau r}(y)$ disjoint from $E_{r}(x)$ has length at least $h(r)$ while $d(u,v)\leq 2\tau r$.

Thus,   Proposition \ref{Probabilistic Lemma}  implies:
$$\frac{\mathcal{G}(u,v,E^{c}_{r}(x))}{\mathcal{G}(u,v)}=P_{u,v}(\gamma \in Traj(u,v):\gamma \cap E_{r}(x)=\emptyset) \leq \phi^{h(r)-2\tau r D}.$$
Applying this estimate for every pair $u,v$ in (\ref{outer}) by (\ref{totalpr}) we  get

$$P_{x,y}(\gamma \in Traj(x,y):\gamma \cap E_{r}(x)=\emptyset, \gamma \cap E_{\tau r}(x) \neq \emptyset, \gamma \cap E_{\tau r}(y) \neq \emptyset)\leq$$

$$\phi^{h(r)-2\tau r D}
P_{x,y}(\gamma \in Traj(x,y):  \gamma \cap E_{\tau r}(x) \neq \emptyset, \gamma \cap E_{\tau r}(y) \neq \emptyset)\leq \phi^{h(r)-2\tau r D}$$

Thus $P_{x,y}(Q_{r}(x,y))\leq 2\phi^{h(r)-2\tau r D}$.

\end{proof}

%\st{Note  in the above proof we applied    {\ref{Probabilistic Lemma}}  not to $P_{x,y}$ but to $P_{u,v}$ separately for each $u\in E_{\tau r}(x), v\in E_{\tau r}(y).$}

We are now ready to prove Theorem \ref{relancona}
\begin{proof}[Proof of Theorem \ref{relancona}]
{Choose $R=R(\delta^{f}(x,y))$ to be larger than the number $R_{0}(\delta^{f}(x,y))$ from Lemma \ref{longpathsunlikely}. By Proposition \ref{Geometric Proposition}, $h(R)/R\to \infty\ (R\to\infty$). Then choosing $R$  sufficiently large we can assume that}
$h(R y)\geq (2D+2)R y\ (\foral y\geq 1).$ Putting $y=\tau^i\ (i\in\Bbb N)$ we obtain
$h(\tau^{i}R)-2\tau^{i}RD\geq 2 \tau^{i}R \geq (i+1)R$ for each $i\geq 0$.

Thus \begin{equation}\label{phibound}\sum^{\infty}_{i=0}\phi^{h(\tau^{i}R)-2\tau^{i} R D}\leq \sum^{\infty}_{{i=0}}\phi^{(i+1)R}=\phi^{R}/(1-{\phi^R})\leq \frac{\epsilon}{4}\end{equation}
 when $R$ is large enough.

{Any} trajectory in $Traj(x,y)$ either passes through $N_{R\hspace*{0.3MM}}o$ or is an element of
$\displaystyle\bigcup^{\infty}_{i=0}Q_{\tau^{i}R}(x,y)$.
{By (\ref{probubound}) and (\ref{phibound}) } {we have
\begin{equation}\label{seriesbound}\begin{aligned}\displaystyle P_{x,y}(\bigcup^{\infty}_{i=0}Q_{\tau^{i}R}(x,y))\leq \sum^{\infty}_{i=0}P_{x,y}(Q_{\tau^{i}R}(x,y))\\  \leq\sum^{\infty}_{i=0}2\phi^{h(\tau^{i}R)-2\tau^{i} R D}\leq \epsilon/2.\end{aligned}\end{equation}
}
Therefore:
\begin{equation}\label{randKarlsson}
P_{x,y}(\gamma\in Traj(x,y):\gamma \cap N_{ R(\delta^{f}(x,y))}o\neq \emptyset)\geq 1-\epsilon/2.
\end{equation}
This completes the proof of Theorem \ref{avoidball}
\end{proof}

\section{Proof of Theorem \ref{Ancona}{.\rm b}: extension to infinite support}

Assume again that $G$ is a nonamenable group and $\mu$ a probability measure on $G$ whose support generates $G$.
The goal of this section is to generalize  Theorem 1.1{.\rm a} to measures with infinite support but superexponential moment: {these are} measures $\mu$ for which
$$\sum_{g\in G: ||g||>N}\mu(g)$$ decays superexponentially in $N$, or equivalently $$\sum_{g\in G} c^{||g||}\mu(g)<\infty$$for all $c>1$.
{We now proceed by proving the analogue Theorem \ref{avoidball} in this context}.
We assume in this section that the Floyd function $f$ decays at least as fast as $x\to x^{-2-c}$ for some $c>0$.
It suffices to only consider  functions of the form $f(x)=x^{-2-c}\ {(c>0)}$. {Indeed the function $R(\cdot)$ is decreasing so} once we prove Theorem \ref{avoidball} for a fixed Floyd function, the analogue for faster decaying Floyd functions follows automatically.
Let $\tau>1$.
We will use the following modification of
Lemma {\ref{helpfunction}}.
\begin{lemma}\label{helpinfinite}
There exists a function $e:\mathbb{R}^{+}\to \mathbb{R}^{+}$ such that as $r\to \infty$ we have
$e(r)\to 0$ and $$\frac{e(r)-e(\tau r)}{r^{2}f(r)}\to \infty$$.
\end{lemma}
\begin{proof}
An easy computation shows $e(r)=\frac{1}{\log(r)}$ does the trick.
\end{proof}

We need the following adaptation of Proposition {\ref{Geometric Proposition}} (this is where we use the assumption on the Floyd function).
\begin{proposition}\label{infiniteavoid}
There are functions $R_{0}:\R\to \R$
and $h:\R\to \R$ with $h(r)/r\to \infty$ as $r\to \infty$
such that for all $x,y \in G$ and all $r>R_{0}(\delta^{f}(x,y))$, for each $u\in E_{\tau r}(x)$ and
$v\in E_{\tau r}(y)$, any path from $u$ to $v$ disjoint from  $E_{r}(x)$ { with jump size bounded by $r/100$ } has length at least $h(r)$.
\end{proposition}\begin{proof}
{It is essentially identical to that of Proposition \ref{Geometric Proposition} with $r/100$ in place of the constant $K$.} {Similar calculations to those of \ref{Geometric Proposition} give $$h(r)={\rm const}\cdot{e(r)-e(\tau r)\over rf(r)}.$$}
\end{proof}

We continue with the proof of Theorem 1.1 b.
For each $n$ we can write $\mu=\mu_{n}+\sigma_{n}$ where
$\mu_{n}$ is the restriction of $\mu$ to the ball
$\{g\in G: ||g||\leq n\}$ and $\sigma_{n}=\mu-\mu_{n}$.

The contribution to $\mathcal{G}(x,y)$ of trajectories of length $M$, with exactly $m$ jumps of size greater than $n$ is
bounded by
$${M \choose m}|\sigma_{n}|^{m}|\mu_{n}|^{M-m}$$
{where for a measure $\sigma$ on $G$ {we use the notation} $|\sigma|=\sigma(G)$.}

Since $\mu$ has superexponential moment, there exists some function
$W: \mathbb{R}\to \mathbb{R}$ with $W(t)/t\to \infty$ as $t\to \infty$ such that for each $K>1$:
\begin{equation} \label{superexp}
K^{W({n})}\sigma_{{n/100}}(G)\to 0
\end{equation}
as ${n}\to \infty$.
{For example, one can take} {$$W(n)=\sqrt{n \log|\sigma_{n/100}|^{-1}},$$
satisfying the above requirement as $\mu$ has superexponential moment.}
% Note, if the estimate (\ref{superexp}) holds for all $K>1$, % the decay there is automatically superexponential for all \& %$K$.

We want to extend the proof of Theorem \ref{avoidball}. Theorem \ref{relancona} will then follow as in the finite support case.
\begin{proof}[Proof of Theorem \ref{avoidball} for measures of superexponential moment]
By equivariance we can assume $z=o$ is the basepoint.
The only step where we need to deviate from the finite support case is estimating for each $u\in E_{\tau r}(x)$ and $v\in E_{\tau r}(x)$ the quantities $\mathcal{G}(u,v,E^{c}_{r}(x))$ and  $\mathcal{G}(u,v,E^{c}_{r}(x))$.

We want to prove {that}
$\mathcal{G}(u,v,E^{c}_{r}(x))\leq \Psi(r)\mathcal{G}(u,v)$ and $\mathcal{G}(u,v,E^{c}_{r}(y))\leq \Psi(r)\mathcal{G}(u,v)$ where
$\Psi(r)$
decays superexponentially in $r$.
If this is true we can proceed as in the proof of Theorem \ref{avoidball}: by decomposing as in (\ref{outer}) and (\ref{totalpr})
{similarly to (\ref{seriesbound})} we {obtain} $P_{x,y}(Q_{r}(x,y))\leq 2 \Psi(r)$.  {Since} $\Psi(r)$ decays super-exponentially we have
$\sum^{\infty}_{i=0}\Psi(\tau^{i}R)<\epsilon$ for large enough $R=R(\epsilon,\delta^{f}_{o}(x,y)).$
Thus $$\mathcal{G}(x,y,B^{c}(o,R))<\sum^{\infty}_{i=0}\mathcal{G}(x,y)\Psi(\tau^{i}r)\leq \epsilon \mathcal{G}(x,y).$$

We estimate $\mathcal{G}(u,v,E^{c}_{r}(x))$; the estimate for $\mathcal{G}(u,v,E^{c}_{r}(y))$ is identical.

{We have} $d(u,v)\leq 2 \tau r${, by Proposition \ref{infiniteavoid}} any trajectory with no jumps of length greater than $r/100$ has length at least ${h(r).}$ By Proposition \ref{Probabilistic Lemma} the contribution to
$\mathcal{G}(u,v,E^{c}_{r}(x))$ of trajectories with no jumps greater than $r/100$ is at most
$\phi^{h(r)-2D\tau r}\mathcal{G}(u,v)$.

Also by Proposition \ref{Probabilistic Lemma}, the contribution of trajectories of length at least $W(r)$ is at most $\phi^{W(r)-2D\tau r}\mathcal{G}(u,v)$.

It remains to control the contribution to $\mathcal{G}(u,v,E^{c}_{r}(x))$ of trajectories of length at most $W(r)$ with at least one jump of size at least $r/100$.

This is bounded above by
\begin{equation}\label{comb}
\sum^{W(r)}_{m=0}\sum^{m-1}_{k=0}{m\choose k}|\sigma_{r/100}|^{m-k}|\mu_{r/100}|^{k}
\end{equation}

(The condition that there is at least one jump of size $\geq r/100$ is reflected in the fact that the inner sum ends with $k=m-1$ rather than $k=m$).

Since $m-k\geq 1$, in the above expression we have
$|\sigma_{r/100}|^{m-k}\leq |\sigma_{r/100}|$ so  (\ref{comb}) is bounded above by

$$|\sigma_{r/100}| \sum^{W(r)}_{m=0}\sum^{m-1}_{k=0}{m\choose k}|\mu_{r/100}|^{k}$$
which via binomial expansion is bounded above by
$$|\sigma_{r/100}| \sum^{W(r)}_{m=0}(1+|\mu_{r/100}|)^{m}\leq |\sigma_{r/100}| \sum^{W(r)}_{m=0}2^{m}\leq 2^{W(r)+1}|\sigma_{r/100}|.$$

{By the Harnack inequality there is a universal $1>\lambda>0$ such that $$\mathcal{G}(u,v)\geq \mathcal{G}(e,e)\lambda^{d(u,v)} \geq \mathcal{G}(e,e)\lambda^{2\tau r}.$$
Thus the contribution to $\mathcal{G}(u,v,E^{c}_{r}(x))$ of trajectories of length at most $W(r)$ with at least one jump of size at least $r/100$ is bounded above by}

{$$2^{W(r)+1}|\sigma_{r/100}|\lambda^{-2\tau r}\mathcal{G}(e,e)^{-1}\mathcal{G}(u,v).$$}

Since {$W(r)/r\to \infty\ (r\to \infty)$ the above quantity is bounded above by ${\rm (\rm const})^{W(r)}|\sigma_{r/100}|$.}

{By (\ref{superexp}) the latter quantity  tends to $0$ superexponentially fast as $r\to \infty$.}

{Putting everything together, we see {that}
$$\mathcal{G}(u,v,E^{c}_{r}(x))\leq \Psi(r)\mathcal{G}(u,v)$$ where $$\Psi(r)=\max{(2^{W(r)+1}\lambda^{-2\tau r}\mathcal{G}(e,e)^{-1}|\sigma_{r/100}|,\phi^{W(r)-2D\tau r}),\phi^{h(r)-2D\tau r})}$$ {tends to zero superexponentially fast as $r\to\infty$.}}
\end{proof}

We will now provide several useful consequences of Theorem \ref{Ancona} for relatively hyperbolic groups.
Suppose $G$ is relatively hyperbolic with respect to a collection $\mathcal{P}$ of subgroups then the following {Proposition provides a characterization of transition points (see the definition in the introduction) in terms of the Floyd function $f.$}

{\begin{proposition}
[\cite{GePoCrelle}, Corollary 5.10]
For each $\epsilon>0$ and $R>0$ there is a number $\delta>0$ such that
if $y$ is  an $(\epsilon,R)$-transition point of a word geodesic from $x$ to $z$ then $\delta^{f}_{y}(x,z)>\delta$.
\end{proposition}}
{As a result, Theorem \ref{Ancona} admits the following corollary for relatively hyperbolic groups (Corollary \ref{relhypAnc} from the introduction).}

{\begin{corollary} \label{relhypAnc2}
Let $ G$ be hyperbolic relative to a collection of subgroups, and let $\mu$ satisfy the conditions of Theorem \ref{Ancona}.
If $x,y,z\in  G$ {is an ordered triple of distinct points belonging to} a word geodesic $\alpha$ and $y$ is  an $(\epsilon,R)$-transition point  then
$$d_{\mathcal{G}}(x,y)+d_{\mathcal{G}}(y,z)\leq d_{\mathcal{G}}(x,z)+A$$ where $A$ depends only on $(\epsilon, R)$,  and $\mu$.\bx
\end{corollary}}
{\begin{remark}
Theorem \ref{Ancona} cannot be extended to measures of exponential moment. Indeed, on any hyperbolic group Gouezel constructed in \cite{G2} a measure with exponential moment for which Corollary 1.3 fails.
\end{remark}}

{The map $F:\Gf\to \Gb$ from the Floyd compactification to the Bowditch compactification (see Section 3) allows one to transfer the Floyd distance $\d^f$
to $\Gb.$
The proof of \cite[Proposition 3.4.6]{Gerasimov} guarantees that the obtained pseudo-distance  is a real distance  $\overline\d^f$ on $\Gb$  which we call {\it shortcut} distance (see \cite[Section 3]{GePoJEMS}). The construction of ${\overline \d}^f$ implies that for every rescaling function $f$ one has the following inequality:
\begin{equation}\label{shortcut}
\foral x,y,v\in\Gf\ : \d^f_v(x,y)\leq {\rm const}\cdot {\overline \d}^f_{F(v)}(F(x), F(y)).
\end{equation}}

{  Since the function $A(\cdot)$ in Theorem \ref{Ancona} is decreasing   the inequality (\ref{shortcut}) implies the following analog of \ref{Ancona} valid on $\Gb$ in terms of  the distance $\overline{\d}^f$.}

  {\begin{corollary} \label{Ancshortcut}
Let $ G$ be hyperbolic relative to a collection of subgroups, and let $\mu$ satisfy the conditions of Theorem \ref{Ancona}.  Then for the same   decreasing function $A:\mathbb{R}^{+}\to \mathbb{R}^{+}$ as in \ref{Ancona} and for all $x,y,z\in G$ one has
\begin{equation}\label{relancshcut} {d_{\mathcal{G}}(x,y)+d_{\mathcal{G}}(y,z)\leq d_{\mathcal{G}}(x,z)+A(\overline{\delta}^{f}_{y}(x,z)).}\end{equation}\bx
\end{corollary}}

\section{A map from the Martin boundary to the Floyd boundary}
As before, we consider a finitely generated nonamenable group $ G$ with a probability measure $\mu$ on $ G$ whose support generates $ G$, and denote by $\mathcal{G}=\mathcal{G}_{\mu}$ the associated Green's function.
Recall the Green metric on $G$ given by $d_{\mathcal{G}}(x,y)=-\log \frac{\mathcal{G}(x,y)}{\mathcal{G}(e,e)}.$
The horofunction compactification of $( G,d_{\mathcal{G}})$ is called the Martin compactification and denoted by $\overline{ G}_{\mathcal M}$.

The {boundary} $$\partial_{\mathcal M} G=\overline{ G}_{\mathcal M}\setminus  G$$ is called the Martin boundary of $( G,\mu)$ \cite{Sawyer}.
This means $\partial_{\mathcal M} G$ consists of all functions
$h: G \to \R$ such that there exists an unbounded sequence $x_{n}\in  G$ with $$h(x)=\lim_{n\to \infty} d_{\mathcal{G}}(x,x_{n})-d_{\mathcal{G}}(o,x_{n})$$ for all $x\in  G$.
The Martin boundary can also often be described in terms of $\mu$ harmonic functions on $( G, \mu)$.

A function $h: G \to \R$ is called $\mu${-}harmonic (or simply harmonic when there is no ambiguity) if for all $x\in  G$,
$$\sum_{g\in  G}h(xg)\mu(g)=h(x).$$

{For $p,q, x\in  G$  we  set
$\Delta(p,q,x)= d_{\mathcal G}(p,x)-d_\mathcal G(q,x)$ and extend it by continuity: for $\alpha\in\partial G_{\mathcal M}$  we let $\displaystyle\Delta(p,q,\alpha)=\lim_{\substack{{x_n\to\alpha}\\{x_n\in G}}} \Delta(p,q, x_n).$

\begin{lemma}\label{Martinharmonic}
If $\mu$ has superexponential moment, then the function defined by  $$\displaystyle K(\cdot,\alpha)=K_\alpha(\cdot)=e^{-\Delta(\cdot,\hspace*{0.4mm} o, \hspace*{0.4mm} \alpha)}=\lim_{x_n\to\alpha}\frac{{\mathcal G}(\cdot, x_n)}{ {\mathcal G}(o, x_n)}$$  is harmonic for all $\alpha\in \overline{ G}_{\mathcal M}$
\end{lemma}
\begin{proof}
When $\mu$ has finite support this is noted by Woess in \cite[Lemma 24.16]{Woess}.

As in section 6, for each $n$ we can write $\mu=\mu_{n}+\sigma_{n}$ where $\mu_{n}$ is the restriction of $\mu$ to $B_{n}(e)$.

Define the linear operator $P=P_{\mu}$ defined on the space $C(G,\mathbb{R})$ of functions $G \to \mathbb{R}$ by $$P\omega(x)=\sum_{y\in G}p(x,y)\omega(y).$$

%Note, $P$ is NOT continuous with respect to the pointwise convergce topology on $C(G,\mathbb{R})$.

Consider a sequence $y_{n}\in G$ converging to $\alpha \in \partial_{\mathcal M}G$. We want to prove {that} $K_{\alpha}$ is $\mu$-harmonic, i.e. $PK_{\alpha}=K_{\alpha}$.
For this, it suffices to show that for every $x\in G$

(1) $$PK_{y_{n}}(x)\to PK_{\alpha}(x)$$

and

(2)$$PK_{y_{n}}(x)\to K_{\alpha}(x)$$

We first prove (2). Let $G_{n}(x)=\mathcal{G}(x,y_{n})$.

Then $$PG_{n}(x)=\sum_{y\in G}p(x,y)\mathcal{G}(y,y_{n})=\mathcal{G}(x,y_{n})=G_{n}(x)$$ if $x\neq y_{n}$ while
$PG_{n}(y_{n})=\mathcal{G}(e,e)-1$ (accounting for the trajectory of length zero).
Since $K_{y_{n}}=G_{n}/\mathcal{G}(o,y_{n})$ we have
$$PK_{y_{n}}(x)=K_{y_{n}}(x)$$ if $x\neq y_{n}$ and $PK_{y_{n}}(y_{n})=\frac{\mathcal{G}(e,e)-1}{\mathcal{G}(o,y_{n}})${.}

{So} for each fixed $x$ we have
$$PK_{y_{n}}(x)\to K_{\alpha}(x)$$ as $n\to \infty$ proving (2).

Now we prove (1)

Note for each $R$, $$P_{\mu}=P_{\mu_{R}}+P_{\sigma_{R}}.$$
Let $\Upsilon_{n}(x)=|K_{y_{n}}(x)-K_{\alpha}(x)|$.

To prove (1) it suffices to show that $P\Upsilon_{n}(x)\to 0$ for all $x\in G$.

Fix $x\in G$.
Note, for each $R$ and $n$,
$P_{\mu}\Upsilon_{n}=P_{\mu_{R}}\Upsilon_{n}+P_{\sigma_{R}}\Upsilon_{n}$.

By the Harnack inequality there is a uniform $C>1$ with
$K_{z}(y)\leq C^{||y||}$ for all $z\in G$
and thus $\Upsilon_{n}(y)\leq 2C^{||y||}$.

Hence, if $R>||x||$ we have

$$P_{\sigma_{R}}\Upsilon_{n}(x) \leq \sum_{y\notin B_{R}(x)}\mu(x^{-1}y) 2C^{||y||}$$ $$=
\sum_{z\notin B_{R}(e)}\mu(z) 2C^{||xz||}\leq \sum_{z\notin B_{R}(e)}2\mu(z) C^{2||z||}\to 0$$ uniformly in $n$ as $R\to \infty$ since $\mu$ has super-exponential moment.

On the other hand, for each fixed $R$ we have that
$$P_{\mu_{R}}\Upsilon_{n}(x)=\sum_{y\in B_{R}(x)}p(x,y)|K_{y_{n}}(y)-K_{\alpha}(y)|\to 0$$ as $n\to \infty$ since $B_{R}(x)$ is finite and $|K_{y_{n}}(y)-K_{\alpha}(y)|\to 0$ for each $y\in B_{R}(x)$.

It follows that for each $R>||x||$
$$\lim \sup_{n\to \infty}P_{\mu}\Upsilon_{n}(x)\leq \lim \sup_{n\to \infty}P_{\mu_{R}}\Upsilon_{n}(x)+ \lim \sup_{n\to \infty}P_{\sigma_{R}}\Upsilon_{n}(x)$$ $$\leq 0+ \sum_{z\notin B_{R}(e)}2\mu(z) C^{2||z||}\to 0$$ as $R\to \infty$.
Thus, we know $PK_{y_{n}}\to PK_{\alpha}$ as $n\to \infty$, proving (1).

\end{proof}

{For the rest of this section we make the following assumptions on the measure $\mu$ and the Floyd function $f$.}

{Assumption 1:
The inequality (1) is satisfied.}

{Assumption 2:
For every $\alpha \in \partial_{\mathcal M}G$, $K_{\alpha}$ is harmonic.}

{By Theorem \ref{Ancona} and Lemma \ref{Martinharmonic} these axioms are satisfied when $\mu$ has finite support, as well as when $\mu$ has superexponential moment and $f(x)\leq x^{-2-c}$ for some $c>0$.}

The following {fact} is well known, but we include {its} proof for completeness.
\begin{lemma}
A nonconstant harmonic function does not attain its extrema on $ G$.
\end{lemma}
\begin{proof}
Suppose a nonconstant harmonic function $h$ attains a maximum at $x\in  G$.
By harmonicity
$$\sum_{g\in  G}h(xg)\mu(g)=h(x).$$
Since $x$ is a maximum for $h$ we have $h(xg)\leq h(x)$ for all $g\in  G$ so as $\mu$ is a probability measure we must have $h(xg)= h(x)$ for all $g\in  G$. Since $ G$ acts transitively on itself this means $h$ is constant. Thus a nonconstant harmonic function does not attain a maximum. If $h$ is a nonconstant harmonic function so is $-h$. Since $-h$ does not attain a maximum on $ G$, $h$ does not attain a minimum.
\end{proof}
We will need the following.
\begin{lemma}\label{harmonicpotential}
There does not exist a positive harmonic function $h$ with $h(x)\leq \mathcal{G}(o,x)$ for all $x\in  G$.
\end{lemma}
\begin{proof}
Suppose $h$ is a positive harmonic function. Since $h$ is positive and does not attain a maximum on $ G$ there is a sequence $x_{n}\in  G$ with $||x_{n}||\to\infty$ and $\lim \inf_{n\to \infty}h(x_{n})>0$.
On the other hand by the Harnack inequality $\mathcal{G}(o,x_{n})\to 0$ as
$n\to \infty$ so we cannot have
$h(x)\leq \mathcal{G}(o,x)$ for all $x\in  G$.
\end{proof}
We are now ready to prove:
\begin{theorem}\label{martintofloyd}
The identity map on $G$ extends to a continuous {equivariant} surjection
$\overline{ G}_{\mathcal M}\to \overline{ G}_{f}$.
The map identifies points in $\partial_{\mathcal M} G$ whose difference is bounded.
\end{theorem}
\begin{proof}
 Fix $c>0$. For $h\in \overline{ G}_{\mathcal M}$ and $n\in \mathbb{N}$ let
$$A_{h,n}=\{\omega \in \overline{ G}_{\mathcal M}: |h(x)-\omega(x)|<c\ \forall x\in B_{n}o \}\ $$
These are clearly open subsets of $\overline{ G}_{\mathcal M}$ containing $h$.
To prove Theorem \ref{martintofloyd} it suffices to show that for each $c>0$ and for all $\epsilon>0$ there is an $n>0$ such that for all $h\in \partial_{\mathcal M} G$ there is a $p\in \partial_{f} G$ with $A_{h,n}\cap  G \subset B_{f}(p,\epsilon)\cap  G$.
Suppose this is not the case.
Then there is an $\epsilon>0$ and $x_{n},y_{n}\in  G$, $h_{n}\in \partial_{\mathcal M} G$ such that $x_{n},y_{n}\in A_{h_{n},n}$ and $\delta^{f}_{o}(x_{n},y_{n})>\epsilon$.
Passing to a subsequence we have $x_{n}\to \alpha \in \partial_{\mathcal M}  G${,} $y_{n}\to {\beta} \in \partial_{\mathcal M} G$ with $|\alpha(x)-{\beta}(x)|<2c$ for all $x$, i.e.
$$e^{-2c}\leq K(x,\alpha)/K(x,\beta)\leq e^{2c}$$
Fix $x\in  G$. For each $n$ we have either $\delta^{f}_{o}(x_{n},x)>\epsilon/2$ or $\delta^{f}_{o}(y_{n},x)>\epsilon/2$. Suppose $\delta^{f}_{o}(x_{n},x)>\epsilon/2$ for infinitely many $n$.
{Then by Theorem \ref{relancona}  $$\mathcal{G}(x,x_{n})\leq S(\d^f_o(x,x_n))\mathcal{G}(x,o)\mathcal{G}(o,x_{n})\leq C\cdot \mathcal{G}(x,o)\cdot \mathcal{G}(o,x_{n}),$$ where
$C=S(\varepsilon/2).$ We have $\displaystyle K(x,x_n)={{\mathcal G}(x, x_n)\over {\mathcal G}(o, x_n)}\leq C\mathcal{G}(x,o)$.}
As $x_{n}\to \alpha$ this implies $K(x,\alpha)\leq C\cdot\mathcal{G}(o,x)$.

Similarly if $\delta^{f}_{o}(x,y_{n})>\epsilon/2$ for infinitely many $n$, we have $K(x,\beta)\leq C\cdot \mathcal{G}(o,x)$ so $K(x,\alpha)\leq C'\cdot \mathcal{G}(o,x)$ where $C'=e^{2c}C$.
It follows that for all $x\in  G$ we have
$$K(x,\alpha)\leq C'\cdot \mathcal{G}(o,x)$$ contradicting Lemma \ref{harmonicpotential}

{We have proved that the identity map ${\rm id}:G\to G$ embeds the neighborhoods  of the  boundary points of $ \Gm=G\sqcup\partial_{\mathcal M}G$ into the neighborhoods of the boundary points of  $\Gf=G\sqcup\partial_fG $. So the identity map   extends to the continuous equivariant map $\pi=\pi^{f}_{\mu}:\overline{ G}_{\mathcal M}\to \overline{ G}_{f}$. The map is necessarily surjective as if $q\in \partial G_f$ and a sequence $x_n\in G$ tends  to  $q$, then  for a subsequence we have $x_{n_k}\to \alpha\in \partial G_\mathcal M\ (k\to\infty)$. By construction $\pi(\alpha)=q.$}
\end{proof}

Let $\pi=\pi^{f}_{\mu}:\overline{ G}_{\mathcal M}\to \overline{ G}_{f}$ be the map constructed in Theorem \ref{martintofloyd}. {Our next goal is to study the fibers of this map over the points of the Floyd boundary $\partial_fG.$ The rest of this section is devoted to proving that the preimage of every conical point in $\partial_fG$ contains only one point. In the next section we study the fibers of $\pi$ over the parabolic points of $\partial_f G.$}

The following is a simple consequence of {the Ancona type inequality given by Theorem \ref{relancona}.}
\begin{lemma}\label{bdryancona}
There is a function $S_{1}:\R^{+}\to \R^{+}$ such that for all $y\in  G$ and $h\in \partial_{\mathcal M} G$ we have
$$K(y,h)\leq S(\delta^{f}_{o}(y,\pi(h))\mathcal{G}(o,y)$$
\end{lemma}
\begin{proof}
Let $x_{n}\in  G$ converge to $h$ in the Martin compactification.
Then by Theorem \ref{martintofloyd} $x_{n}\to \pi(h)$ in the Floyd topology so for large enough $n$ we have
$$\delta^{f}_{o}(y,x_{n})\geq \delta^{f}_{o}(y,\pi(h))/2$$

By Theorem \ref{relancona} it follows that
$$\mathcal{G}(y,x_{n})/\mathcal{G}(o,x_{n})\leq S(\delta^{f}_{o}(y,x_{n}))\mathcal{G}(y,o)\leq S(\delta^{f}_{o}(y,\pi(h))/2)\mathcal{G}(y,o)$$

and taking limits gives
$$K(y,h)\leq S_{1}(\delta^{f}_{o}(y,\pi(h))\mathcal{G}(y,o)$$ for $S_{1}(t)=S(t/2)$
\end{proof}
For a function $Q: G\to \R_{\geq 0}$ define its Martin support to be
$$supp_{\mathcal M}Q=\{\zeta \in \partial_{\mathcal M} G:\lim \sup_{x\to \zeta}Q(x)>0 \}\ $$
 and its Floyd support
$$supp_{f}Q=\{{q} \in \partial_{f} G:\lim \sup_{x\to {q}}Q(x)>0 \}. $$
Note $supp_{f}Q=\pi(supp_{\mathcal M}Q)$.
Clearly if $0\leq u\leq h$ then $supp_{f}u\subset supp_{f}h$ and
$supp_{\mathcal M}u\subset supp_{\mathcal M}h$.

\begin{lemma}\label{kerzer}
Let $A_{1},A_{2}\subset \partial_{\mathcal M} G$ be closed subsets of the Martin boundary such that $\pi(A_{i})$ are disjoint subsets of the Floyd boundary.
Then for any sequence $x_{n}\to \alpha$ with $\alpha \in A_{1}$
the functions $\beta \to K(x_{n},\beta)$ converge  to $0$ uniformly over $\beta\in A_{2}$.
 \end{lemma}

\begin{proof}
Let $U_{i}$ be closed neighborhoods of $A_{i}$ in $ {\Gm}$ such that
$\pi(U_{i})$ are disjoint. Then there is a $d>0$ such that ${\delta^{f}_{o}}(u_{1},u_{2})>d$ for all $u_{i}\in U_{i}$.

By Theorem \ref{relancona} there is a $C=C(d)>0$ such that $K(u_{1},u_{2})<C\cdot \mathcal{G}(u_{1},o)$ for all $u_{i}\in U_{i}\cap G$.

Passing to the limit $u_2\to \beta\in A_2$ we obtain $K(u_{1},\beta)<C{\cdot}\mathcal{G}(u_{1},o)$ for any $u_{1}\in U_{1}$.

Now, suppose $x_{n}\to \alpha$ with $\alpha \in A_{1}$. Then $x_{n}\in U_{1}$ for large enough $n$.
Thus, for each $\beta\in A_{2}$ we have  $K(x_{n},\beta)<C{\cdot}\mathcal{G}(x_{n},o)\to 0$.
\end{proof}
A positive $\mu$-harmonic function $h: G \to \R_{+}$ is called minimal harmonic if for every $\mu$-harmonic function $q: G \to \R_{+}$ with $q\leq h$ we have $q=c{\cdot} h$ for some constant $c\in \R$.

The following is the Martin representation theorem, see e.g. \cite{Sawyer}.

Let the minimal Martin boundary $\partial^{min}_{\mathcal M} G \subset \partial_{\mathcal M} G$ consist of those $\alpha \in  \partial_{\mathcal M} G$ for which $K(.,\alpha)$ is minimal.

\begin{theorem} [Martin Representation Theorem]\label{martinrep}
Any minimal harmonic function $h: G \to \R_{+}$ with
$h(o)=1$ is of the form $h(x)=K(x,\alpha)$ for some
$\alpha \in \partial_{\mathcal M} G$.
For any positive $\mu$-harmonic function $h: G \to \R_{+}$ there is a finite measure $\nu^{h}$ on $\partial^{min}_{\mathcal M} G $ such that
$$h(x)=\int_{\alpha \in \partial^{min}_{\mathcal M} G}K(x,\alpha)d\nu^{{h}}({\alpha})$$ for every $x\in  G$.
\end{theorem}

\begin{proposition}
Let $h$ be any positive harmonic function.
Then the representing measure $\nu^{h}$ is supported on $\pi^{-1}\overline{supp_{f}h}$.
\end{proposition}
\begin{proof}
Suppose not. Then there is a closed subset $A\subset \partial_{\mathcal M} G\setminus \pi^{-1}\overline{supp_{f}h}$ with $\nu^{h}(A)>0$. {Consider the positive harmonic function}
$$h'(x)=\int_{{\gamma} \in A}K(x,{\gamma})d\nu^{h}({\gamma}).$$ {By the Martin representation theorem the set  $\partial^{min}_{\mathcal M} G$ is a subset of $\partial_{\mathcal M}G$ of  full $\nu^h$-measure}, so we have  $h'\leq h$ everywhere.

Since $h'$ cannot attain its maximal value on $ G$ there is a sequence $x_{n}\in  G$ converging to some $\beta\in \partial_{\mathcal M} G$ with $h'(x_{n})\to c=supp_{\mathcal M} h'>0$.
This implies $\lim \inf h(x_{n})\geq c>0$ so $\beta \in \pi^{-1}supp_{f}h$.
Since $A$ is a closed set disjoint from the closure of $\pi^{-1}supp_{f}h$ we get  by Lemma \ref{kerzer} that
$K(x_{n},{\gamma})\to 0$ uniformly for ${\gamma} \in A$.
This implies $h'(x_{n})=\int_{{\gamma}\in A}K(x_{n},{\gamma})d\nu^{h}({\gamma})\to 0\ (x_n\to\beta)$ contradicting
$h'(x_{n})\to c>0$.
\end{proof}
\begin{corollary}\label{floydsupport}
For every $\alpha\in \partial_{\mathcal M} G$, if $h=K(.,\alpha)$ then $\nu^{h}$ is supported on $\pi^{-1}(\pi(\alpha))$.
\end{corollary}
\begin{corollary}\label{minimalfunctions}
For every $\zeta \in \partial_{f} G$, $\pi^{-1}\zeta$ contains a point of $\partial^{m}_{\mathcal M} G$.
\end{corollary}
\begin{proof}
{Since the map $\pi$ is surjective} there exists  $h\in \pi^{-1}(\zeta)$. Then by Corollary \ref{floydsupport}  $\nu^{h}$ gives full (hence nonzero) measure to $\pi^{-1}(\zeta)\cap \partial^{m}_{\mathcal M} G$ so this set must be nonempty.
\end{proof}
\begin{corollary}\label{quasiconicalinjective}
If $\zeta \in \partial_{f} G$ is a point such that there is a $C>0$ with $K(x,\beta)/K(x,\alpha)\leq C$ for all $x\in G$ and $\alpha,\beta \in \pi^{-1}\zeta$ then $\pi^{-1}\zeta$ consists of a single point.
\end{corollary}
\begin{proof}
Let $\alpha \in \pi^{-1}\zeta$ be such that $h=K(.,\alpha)$ is minimal and $\beta \in \pi^{-1}\zeta$ be arbitrary. By assumption $K(x,\beta)/K(x,\alpha)\leq C$ for all $x\in  G$ and thus by minimality of $K(.,\alpha)$ we have $K(x,\alpha)=c\cdot K(x,\beta)$ for all $x\in  G$ for some constant $c$. By definition of the Martin boundary, we must have $K(o,\alpha)=K(o,\beta)=1$ so $c=1$ thus $\alpha=\beta$.
\end{proof}

We will use Corollary \ref{quasiconicalinjective} to prove that if $\zeta \in \partial_{f} G$ is conical, {then} $\pi^{-1}(\zeta)$ consists of a single point.

\begin{proposition}\label{pseudocon}
Assume $\zeta \in \partial_{f} G$ is conical.
Then there is a constant $C=C_{\zeta}$ such that for each $x\in  G$ there exists a neighborhood $P_{x}(\zeta)$  of $\zeta$ in  $\partial_{f} G$ for which one has $$C^{-1} \leq K(x,p)/K(x,q)\leq C$$ for all $p,q\in P_{x}(\zeta).$
\end{proposition}

\begin{proof}

Let $g_{n}\in  G$ and distinct points $\alpha,\beta  \in \partial_{f} G$ be such that
$g_{n}\zeta \to \alpha$ and $g_{n}\eta \to \beta$ for all $\eta \in ( G \cup \partial_{f} G)\setminus\{\zeta\}.$

Let $U,V\subset  G \cup \partial_{f} G$ be disjoint closed neighborhoods of $\alpha$ and $\beta$ respectively and $0<\epsilon<\delta^{f}_{o}(U,V)$.

Fix $x,y\in  G$, then for
 $n$ large enough and $s=g_{n}$ we have $sx,sy \in V$ and $s\zeta \in U.$
Let $P_{x,y}(\zeta)=s^{-1}U$.
Then $P_{x,y}(\zeta)$ is a closed neighborhood of $\zeta$ in $ G \cup \partial_{f} G$.
For $p\in P_{x,y}(\zeta)$ we have $sp\in U$ and $sy,sx \in V$ so $\delta^{f}_{o}(sp,sx)>\epsilon$ and $\delta^{f}_{o}(sp,sy)>\epsilon$. Thus, $\delta^{f}_{s^{-1}o}(p,x)>\epsilon$ and $\delta^{f}_{s^{-1}o}(p,y)>\epsilon$.
Hence there is a constant $C=C_{\epsilon}$ such that by the Harnack inequality   and Theorem \ref{anconamult} we obtain
$$\mathcal{G}(p,s^{-1}o)\mathcal{G}(s^{-1}o,y)\leq \mathcal{G}(p,y)\leq C \mathcal{G}(p,s^{-1}o)\mathcal{G}(s^{-1}o,y)$$ and
$$\mathcal{G}(p,s^{-1}o)\mathcal{G}(s^{-1}o,x)\leq \mathcal{G}(p,x)\leq C \mathcal{G}(p,s^{-1}o)\mathcal{G}(s^{-1}o,x)$$ for all $p\in P_{x,y}(\zeta)$.

Hence,
$$ C^{-1} \cdot  \frac{\mathcal{G}(s^{-1}o,x)} {\mathcal{G}(s^{-1}o,y)} \leq \frac{\mathcal{G}(x,p)}{\mathcal{G}(y,p)}\leq C\cdot \frac{ \mathcal{G}(s^{-1}o,x)}{ \mathcal{G}(s^{-1}o,y)}.$$

This is true for every $p\in P_{x,y}(\zeta)$ hence for distinct $p,q\in P_{x,y}(\zeta)$ we have
$$C^{-4}\leq\frac{\mathcal{G}(x,p)/\mathcal{G}(y,p)}{\mathcal{G}(x,q)/\mathcal{G}(y,q)}\leq C^{4}$$
In particular, letting $y=o$ and $P_{x}=P_{x,o}$ we have $$D^{-1}\leq K(x,p)/K(x,q)\leq D$$ for all $p,q\in P_{x}$ where $D$ is a constant.
\end{proof}

\begin{corollary}
For each conical $\zeta\in \partial_{f} G$ there is a constant $D=D(\zeta)$ such that for all $\alpha,\beta \in \pi^{-1}\zeta$ and $x\in  G$ we have $K(x,\alpha)/K(x,\beta)\leq D$.
\end{corollary}

\begin{proof}
Let $p_{n},q_{n}\in  G$ with $p_{n},q_{n}\to \zeta$ in the Floyd compactification and $p_{n}\to \alpha$, $q_{n}\to \beta$ in the Martin compactification. Then by Proposition \ref{pseudocon}
for each $x\in  G$ we have a neighborhood  $P_x\subset  G \cup \partial_{f} G$ of $\zeta$ such that $D^{-1} \leq K(x,p)/K(x,q)\leq D$ for all $p,q\in P_x$ and some uniform constant D.
Then for large enough $n$ we have $p_{n}, q_{n}\in P_x,$ so  $D^{-1} \leq K(x,p_{n})/K(x,q_{n})\leq D$.
Passing to the limits   $p_n\to\alpha, q_n\to \beta\ (n\to\infty)$ in the Martin boundary, we obtain the result. \end{proof}
\begin{corollary}\label{conicalinjective}
If $\zeta \in \partial_{f} G$ is conical, $\pi^{-1}(\zeta)$ consists of a single point.
\end{corollary}

\section{Preimage of the limit set of a subgroup acting cocompactly outside of it}
Let $G$ be a finitely generated group and $\partial_fG$ denotes its Floyd boundary with respect to the rescaling function $f$.  We also denote by $X$   the Floyd compactification   $\GAF$
of $G$ which is the union $G\sqcup\partial_fG$ (where we identify $G$ with the set of vertices of the Cayley graph of $G$).

 Without loss of generality we can assume that the Floyd boundary is not a point; otherwise the results below become trivial. Recall {that if} $G$ is  relatively hyperbolic then there exists  ${\nu}\in (0,1)$ such that then $\pfg$ is not trivial for every rescaling function $f$ satisfying $f(n)\leq{\nu}^n\ (n\in \Bbb N)$  \cite{Gerasimov}. {We denote by $\delta_v=\delta_v^f$ the Floyd distance   for a fixed rescaling function $f$ and based at the vertex $v$ of the graph.} By \cite{Karl} the action $G\act\GAF$ is a convergence action. For a subgroup $H< G$ we denote by $\Lambda H$ its limit set for the action on $\GAF.$  Since the action is convergence $\La H$  coincides with the boundary $\partial^fH$ of the orbit $H$ in $\Gf$.
 %The upper index is put to distinguish it from its own Floyd boundary $\partial_fH$. Notice that $\plf H$ and $\psf H$ are known to coincide in some cases (see \cite[Corollary 7.8]{GePoCrelle})  and it is  not known in many others.
 If there is no ambiguity   we keep the notation  $\La H$ for the boundary $\psf H.$

 We consider geodesics (infinite or not) in the Cayley graph equipped with the word metric $d(\cdot,\cdot)$. Denote by  $\ch$  the convex hull   $$\displaystyle \{{\gamma}:\Z\to G\  {\rm is\ a\ geodesic}\ :\  \lim_{n\to\pm\infty}\ga(n)\in\La H\}$$ of the limit set $\La H$ in $X.$

Let   $\pmg$  be the Martin boundary of $G$ with respect to a symmetric measure $\mu$ on $G$ satisfying Assumption 1 and 2 {and $\Gm=\Gam$ its Martin compactification.} Let  $\psm H$ be the topological boundary
 of $H$ in $\Gm$, i.e. the set of accumulation points of $H$  in $\Gm$.

  A subgroup $H$ of $G$ is called quasiconvex if any quasigeodesic between two elements of $H$ belongs to a uniform neighborhood of $H.$ It is called {\it fully quasiconvex} if it is quasiconvex and every parabolic subgroup $P$ of $G$ either intersects $H$ in a subgroup having finite index in $P$ or is finite. By \cite[Theorem B]{GePoCrelle} the cocompactness of the action $H$ on $X\setminus\Lambda H$ is equivalent to the full quasiconvexity of $H$ in a relatively hyperbolic group $G.$ By Corollary {\ref{minimalfunctions}}}
for every point $\xi\in\partial_f G$ its preimage $\pi^{-1}(\xi)$ contains {points} from the minimal Martin boundary
$\partial^{m}_{\mathcal M} G.$ The aim of this section is the following proposition refining this statement
for the limit points of the fully quasiconvex subgroups.

\begin{proposition}
\label{lsp}
Let $\pi:\pmg\to\pfg$ be a continuous equivariant map  from the Martin boundary  to the Floyd boundary of $G.$ Let $H< G$ be a subgroup acting cocompactly on $X\setminus\Lambda H$.
Then \begin{equation}\label{preimbound} \pi^{-1}(\Lambda H)\cap \partial^{m}_{\mathcal M} G\subseteq\psm H\end{equation}
\end{proposition}

\medskip

{\it Remark.} There exist relatively hyperbolic groups with symmetric finitely supported measures whose minimal Martin boundary is a proper subset of the Martin boundary. Indeed, suppose $G_{1}$ is nonamenable, $G_{2}$ any finitely generated infinite group, and $\mu_{i}$ finitely supported generating measures on $G_{i}$.
Let $G=G_{1}\times G_{2}$ be the Cartesian product and $\mu=\mu_{1}\times \mu_{2}$ be the product measure.
Picardello and Woess show \cite[Corollary 4.4]{PW} {that} the Martin boundary of $(G,t\mu)$ contains non-minimal points for any $t$ up to and including the inverse of the spectral radius of $\mu$.

Then Theorems 26.18 and 26.21 of \cite{Woess} imply that whenever $({\Gamma},m)$ is any finitely generated group and $m$ a finitely supported measure on ${\Gamma},$  the Martin boundary of the free product $(G*{\Gamma},\mu+m)$ contains non-minimal points. \footnote{\small\sf {We thank Wolfgang Woess for explaining us this example.}}

\medskip
\noindent {\it Proof of Proposition \ref{lsp}.} In all arguments below   the subgroup $H$ acting cocompactly on $X\setminus\Lambda H$ is fixed. For a vertex $x\in G$ we denote by ${\s Pr}_{\ch} x$ the   projection set  $\{y\in\ch\ :\ d(y,x)\leq d(x,\ch)\}$ of   $x$ to $\ch.$

\begin{lemma}\label{geodep} There exist two   constants  $D=D(H)<+\infty$ and $\delta=\delta(H)>0$ such that for every    sequence $x_n$ converging to a point $\eta\in\La H$  and every vertex $o\in G$  for the sequence  $o_n\in{\s Pr}_{\ch} x_n$  we have $\delta^f_{o_n}(o, x_n)
\geq \delta$ and  $d(o_n, \gamma_n)\leq D\ (n >n_0)$ where $ \gamma_n=[o, x_n]$ is a geodesic between $o$ and $x_n$.  \end{lemma}

\proof The set $\ch\cup\La H$ is closed subset of $X.$ Since the action of $H$  on $X\setminus\Lambda H$ is   cocompact,   the quotient $\ch/H$ is   finite \cite[Proposition 4.5]{GePoCrelle} (an $H$-invariant set $\ch$ having such a property is called  {\it weakly homogeneous} in \cite{GePoCrelle}).
Let    $\c F$  denote a compact fundamental set    for the action of $H$ on $X\setminus\Lambda H$. Then there exists a constant {$\nu=\nu(H) >0$} such that $\delta_1(\c F, \La H)\geq \nu$ where $\delta^f_1$ is the Floyd metric based at $1\in G$.

  Let $F={\s Pr}_{\ch} (\c F\cap G)$. Since $\ch$ is $H$-invariant and weakly homogeneous by \cite[Proposition 3.5]{GePoCrelle} the
 diameter $d={\rm diam}(F)$ with respect to the word metric  is finite and depends only on the constant {$\nu$} above.

Let $\ga_n:\N\to G$ be a geodesic between $o$ and $x_n$ such that $\displaystyle\lim_{n\to\infty}x_n=\eta\in \La H.$ Then there exists a sequence  $h_n\in H$ such that $y_n=h_n(x_n)\in \c F\cap G$. Since the action of $H$ on the Cayley graph of $G$ is isometric, the images $h_n(o_n)$ of the projections $o_n$ of $x_n$ to $\ch$ are projections of $y_n$ to $\ch$. So $h_n(o_n)\in F$ and $d(v, h_n(o_n))\leq d$ for a fixed point $v\in F$.

Set $z_n=h_n(o).$
 and fix a sufficiently small $\ve\in ]0, \delta[.$ Denote by $\s N^{{f}}_\ve(\La H)$ the $\ve$-neighbourhood of $\La H$ in $X$ with respect to the Floyd distance $\delta^{{f}}_{{1}}.$  {We have}  $z_n\in N^{{f}}_\ve(\La H)$ for $n>n_0$.

 Using the inequality {(\ref{coef})}  we obtain \begin{equation}\label{flun}\delta^{{f}}_v(\c F,\s N_\ve(\La H))\geq {\nu-\ve\over {\kappa}^{d(1,F)}}>0.\end{equation} Since  $F$ is a finite set depending on the subgroup $H$ only, {the} above lower bound depends on  $H$ and fixed $\ve \in]0, \nu[$. Applying $h_n^{-1}$ to  (\ref{flun}) we obtain a constant $\delta=\delta(H)>0$ for which  $\delta^f_{o_n}(o, x_n)\geq \delta >0$.

By Karlsson lemma \cite[Lemma 1]{Karl} there exists a constant $D=D(H, \ve)$ such that $d(o_n, \ga_{{n}})\leq D.$ \qed

\bigskip

 {\sf End of the proof of  Proposition \ref{lsp}}. {By Corollary \ref{minimalfunctions}}
For a point  $\xi\in \Lambda H$ fix a  point $\alpha\in  \pi^{-1}(\xi)\cap \partial^{m}_{\mathcal M} G$ such that the harmonic function $K_{\alpha}$ is minimal. Consider a sequence $x_n\to \alpha\ (n\to\infty)$ and their projections $o_n\in {\s Pr}_{\ch} (x_n)$ to $\ch.$

For a geodesic ${\beta}_n=[1, x_n]$    by Lemma \ref{geodep} we  obtain points $w_n\in {\beta}_n$ such that $d(o_n, w_n)=d(o_n, {\beta}_n)\leq D.$ Then applying {the}  Harnack inequality (Lemma 2.1) {for any $x\in G$ we have}

\begin{equation}\label{kan}{K_{o_n}x\over K_{w_n}x}= {\mathcal{G}(x,o_n)\cdot \mathcal{G}(1, w_n)\over \mathcal{G}(1,o_n)\cdot \mathcal{G}(x, w_n)}\leq \lambda^{-2d(w_n, o_n)}\leq \lambda^{-2D}.\end{equation}

We also have

\begin{equation} \label{AKgeod}{K_{w_n}x\over K_{x_n}x}={\mathcal{G}(x, w_n)\cdot \mathcal{G}(1,x_n)\over \mathcal{G}(1, w_n)\cdot \mathcal{G}(x, x_n)}  \leq A(\delta^f_{w_n}(1, x_n)).\end{equation}

Indeed, in the nominator of (\ref{AKgeod}) by  { Theorem \ref{relancona} we have}:

\centerline{$\mathcal{G}(1,x_n)\leq {S}(\delta^f_{w_n}(1,x_n))\cdot \mathcal{G}(1,w_n)\cdot  \mathcal{G}(w_n, x_n)$;}

\noindent and in the denominator we used the (triangle) inequality
$\mathcal{G}(x, x_n)\geq \mathcal{G}(x, w_n)\cdot \mathcal{G}(w_n, x_n).$

By Lemma \ref{geodep} $\delta^f_{o_n}(1,x_n)\geq\delta$ and $d(o_n, w_n)\leq D$ so $\delta^f_{w_n}(1,x_n)\geq\lambda^{-D}\cdot \delta$   which is a uniform constant too.
The function ${S}(\cdot)$ is decreasing
so (\ref{kan}) and (\ref{AKgeod}) imply \begin{equation}\label{upbd}{K_{o_n}x \over K_{x_n}x}\leq C,\  {\rm where}\  C=A(\lambda^{-D}\delta)\cdot \lambda^{-2D}.\end{equation}

Replacing  in the previous argument  the geodesic   $[1, x_n]$ by a geodesic  $[x, x_n]$ we similarly  obtain the points $w_n\in [x, x_n]$ such that for the  projections $o_n\in {\s Pr}_{\ch} x_n$  we  have $d(w_n, o_n)\leq D$ and $\delta^f_{o_n}({x, x_n})\geq \delta$ for  the same constants $D$ and $\delta$ from Lemma \ref{geodep}. {Then} the previous argument implies  the double inequality:

\begin{equation}\label{doublein} {1\over C}\leq {K_{o_n}x\over K_{x_n}x}\leq  C,
\end{equation}
where $C$ is as in (\ref{upbd}).

Up to passing to a subsequence we can assume that the sequence $o_n\in \ch$ converges to some point $\beta\in \partial_{\mathcal M} G$.   From (\ref{doublein})
 we obtain \begin{equation}\label{doublebd}{1\over C}\leq {{K_{\beta}x\over K_{\alpha}x}}\leq  C.\end{equation}
 Then ${K_\beta}\leq C\cdot {K_\alpha}$ and so ${K_\beta}=C\cdot {K_\alpha}$ by  minimality of $\alpha$.

 We have that $\displaystyle\alpha=\lim_{n\to\infty} o_n={\beta}$ and $o_n\in \ch.$  Since $H$ is quasiconvex (\cite[Proposition 4.5]{GePoCrelle}) there exists a constant $C_1$ such that for every $o_n\in \ch$ there exists $o'_n\in H$ such that $d(o_n, o'_n)\leq C_1$. Applying again the Harnack inequality we obtain  $\displaystyle {K_{o'_n}x\over K_{o_n}x}\leq C'$.
 Since $\alpha\in \partial_{\mathcal M}^m G$ is minimal by the same argument as above we also have $\displaystyle\lim_{n\to\infty}o'_n=\alpha.$
 We have proved that every minimal point in $\pi^{-1}(\Lambda H)$ is an accumulation point of the $H$-orbit. The Proposition is proved.\bx

\bigskip

{\it Remarks.} 1. In the proof above we use the set $\ch$ instead of $H$ (i.e. an $H$-horosphere instead of $H$) because the action of $G$ on the set of horospheres (by the right multiplication) preserves all distances, unlike the action by conjugation. This invariance will be used in Corollary \ref{ppoint} below.

\medskip

2. One can also notice that the choice of approximation sequence $(o_n)\subset \ch$  as   projection  of the approximating sequence $(x_n)\subset G$ is constructive. One can prove that
$\displaystyle\lim_{n\to\infty} \pi(o_n)=\lim_{n\to\infty} \pi(x_n)=\xi\in \Lambda H$ without assuming that {the} limit point $\alpha$ on the Martin boundary is minimal.   Indeed, if it is not the case, the word distance $d(o_n, x_n)$ is unbounded. By Lemma \ref{geodep} there exists a point $u_n\in [x_n,\xi[$ such that $d(o_n, u_n)\leq D$. Then since $\pi(x_n)\to\xi$ we obtain that the infinite geodesic rays $[x_n, \xi[$ converge to a horocycle $l$ based at $\xi$ (implying  by \cite[Lemma 3.6]{GePoJEMS}  that $\xi$ is a parabolic point). But $l\subset \ch$ so $d(x_n, o_n) > d(x_n, \ch)\ (n> n_0)$ which is impossible by definition of
$o_n$. However this argument does not give a uniform estimate for $d(o_n, x_n)$ and therefore we needed to use the   inequality {(\ref{relanc})} instead.

 \bigskip

{Consider now a minimal geometrically finite action of $G$ on a compactum $T$.}
By Theorem \ref{martintofloyd} there exists an equivariant continuous map $\pi : \Gm\to\Gf$ from the Martin to the Floyd compactification. There also exists an equivariant continuous (Floyd) map $F$ from the Floyd compactification $\Gf$ to the Bowditch compactification $\Gb=G\sqcup T$ \cite{Gerasimov}. So we have an equivariant continuous map   $\varphi=F\circ\pi:\Gm\to\Gb.$

\begin{corollary}\label{ppoint} Let $p\in T$ be a bounded parabolic point and $H$ the stabilizer of $p$ for the action $G\curvearrowright T$. Then
the {inclusion} (\ref{preimbound}) is satisfied for the map $\varphi$:

\medskip

\begin{equation}\label{embedparab}\varphi^{-1}(T)\cap \partial^{m}_{\mathcal M} G\subseteq\psm H.\end{equation}

\medskip

Furthermore there exists a uniform constant $C>0$ such that for every bounded parabolic point
 $p\in T$ and {every} $\alpha \in \varphi^{-1}(p)$ {there is some $\beta \in \partial^{\mathcal M}H$ such that}

 $$C^{-1}\leq K_{\alpha}/K_{\beta}\leq C.$$
 \end{corollary}

 \proof We need to show that the constant $C$ can be chosen uniformly being not depending on a parabolic point. Indeed for every parabolic point $p\in T$ the action of its stabilizer $H$ on $T$ is cocompact on $T\setminus\{p\}$. Then $F^{-1}(p)$ is the limit set $\Lambda_fH$ for the action $H\act \Gf$ \cite[Theorem A]{GePoJEMS}. Consequently $(F^{-1} (p))^c=\partial_f G\setminus \plf H.$ Since $F$ is equivariant and continuous and $\partial_fG$ is compact, $H$ acts cocompactly on $(F^{-1} (p))^c$.
So  by Proposition \ref{lsp} we obtain  the inequality (\ref{doublebd}) where $\displaystyle\alpha=\lim_{n\to\infty} o_n\in{\partial_{\mathcal M} H.}$

The constant $C$ found in Proposition \ref{lsp} depends only on the subgroup $H$. Furthermore the system of all horospheres
$\{\ch_H\ :\ H$ is maximal parabolic subgroup for the action $G\act T\}$
 is $G$-invariant and contains at most finitely many $G$-non-equivalent horospheres \cite[Main Theorem.a]{Ge1}. Since
$\delta^f_v(x,y)=\delta^f_{gv}(gx, gy)$ and $d(gx, gy)=d(x,y)\ (g\in G)$  the constant $C$ is the same for the conjugacy class of the maximal parabolic subgroup $H$. Since there are at most finitely many such classes {\cite{Ge1},}   it can be chosen uniformly for all maximal parabolic subgroups of $G$ for the geometrically finite action $G\act T.$ \qed

\medskip

Here are several open questions motivated by the above discussion:

\medskip

\noindent {\bf Questions.} {Let $H<G$ be a fully quasiconvex subgroup of a relatively hyperbolic group $G.$} \begin{enumerate}
 \medskip

 \item [a)] Is $\psm H={\varphi}^{-1}(\Lambda H) ?$

\item [b)] Does the inequality (\ref{doublebd})  imply that the points $\alpha$ and $\beta$ give rise to the same point
at the Martin boundary of $G$ (without assuming the minimality of one them)?
 \end{enumerate}

Note that  $b)\Rightarrow a)$ by the proof of Proposition \ref{lsp}.

{We also note that by the existence of the continuous extension $\pi:\Gm\to \Gf$ of the identity map ${\rm id}:G\to G$ (Theorem \ref{martintofloyd}),
besides the inclusion (\ref{embedparab}) we also have $\psm H\subseteq \pi^{-1}(\Lambda H).$ It is not clear at this moment how to obtain the opposite inclusion.}\bx
 \bx

\bigskip

\section{Some recurrence properties for random walks on convergence groups}
In this section, independent from the previous ones, we review and prove some results on random walks on groups which we will need in Sections 9 and 10.
Let $G$ be an infinite group.
Let $\mu$ be a probability measure on $G$ and
let $\mu^{\mathbb{Z}}$ be the product measure on $G^{\mathbb{Z}}$.

Let $T:G^{\mathbb{Z}}\to G^{\mathbb{Z}}$ be the following invertible
transformation:
$T$ takes the two-sided sequence $(h_{i})_{i\in \mathbb{Z}}$
to the sequence $(\omega_{i})_{i\in \mathbb{Z}}$ with
$\omega_{0}=e$ and $g_{n}=g_{n-1}h_n$ for $n\neq 0$.
Explicitly, this means
$$\omega_{n}=h_{1}\cdots h_{n}\quad \text{ for }n>0$$ and
$$\omega_{n}=h^{-1}_{0}h^{-1}_{-1}\cdots h^{-1}_{-n+1}\quad \text{ for }n<0.$$

Similarly, let $\mu^{\mathbb{N}}$ be the product measure on $G^{\mathbb{N}}$.
Let $T_{+}:G^{\mathbb{N}}\to G^{\mathbb{N}}$ be the transformation that
takes the one-sided infinite sequence $(h_{i})_{i\in \mathbb{N}}$
to the sequence $(\omega_{i})_{i\in \mathbb{N}}$ with
$\omega_{0}=e$ and $\omega_{n}=\omega_{n-1}h_n$ for $n\neq 0$.
Explicitly, for $n>0$ this means
$$\omega_{n}=h_{1}\cdots  h_{n}.$$

Let $\overline{P}$ be the pushforward measure
$T_{*}\mu^{\mathbb{Z}}$ and $P$ the pushforward measure
$T_{+*}\mu^{\mathbb{N}}$.

The measure $P$ describes the distribution of $\mu$
sample paths, i.e. of products of independent $\mu$-distributed increments.
Let $\hat{\mu}$ be the measure on $G$ given by $\hat{\mu}(g)=\mu(g^{-1})$.
Let $\hat{P}$ be the pushforward measure
$T_{+*}\hat{\mu}^{\mathbb{N}}$.
The measure space $(G^{\mathbb{Z}},\overline{P})$ is
naturally isomorphic to $(G^{\mathbb{N}},P)\otimes (G^{\mathbb{N}},\hat{P})$
via the map sending the bilateral path $\omega$ to the pair of unilateral
paths
$((\omega_{n})_{n\in \mathbb{N}},(\omega_{-n})_{n\in \mathbb{N}})$.

Let $\sigma: G^{\mathbb{Z}}\to G^{\mathbb{Z}}$ be the left Bernoulli shift:
$\sigma(\omega)_{n}=\omega_{n+1}.$
By basic symbolic dynamics (see e.g. \cite{Furman}), $\sigma$ is invertible,
measure preserving and ergodic with respect to $\mu^{\Z}$.
Therefore, when restricted to sequences with $e$ at the $0th$ coordinate,
\[U=T\circ \sigma \circ T^{-1}\] is invertible, measure preserving and
ergodic with respect to $\overline{P}$.
Note that for each $n\in \mathbb{Z}$,
$$(U\omega)_{n}=\omega^{-1}_{1}\omega_{n+1}$$
and more generally
$$(U^{k}\omega)_{n}=\omega^{-1}_{k}\omega_{n+k}.$$

We will use the following result of Guivarch \cite{Guivarch}.
\begin{theorem}
Let $G$ be any countable group, let $\mu$ be any measure on $G$ whose support generates a nonamenable subgroup, and let $d$ be any proper left invariant metric on $G$. Then for $P$ a.e. $\omega \in G^{\N}$ we have $\lim\inf_{n\to \infty} \frac{d(\omega_{n},e)}{n}>0$, so in particular $d(\omega_{n},e)\to \infty$.
\end{theorem}

Suppose $G$ acts continuously on an infinite compact Hausdorff space $B$.
A Borel probability measure $\nu$ on $B$ is called $(G,\mu)${-}stationary if $$\nu(A)=\sum_{g\in G}\nu(g^{-1}A)\mu(g)$$ for all Borel $A\subset B$.
The following is classical.
\begin{proposition}
If $G$ acts continuously on a compact Hausdorff space $X$ and $\mu$ any probability measure on $G$ then there is some $(G,\mu)${-}stationary measure $\nu$ on $G$.
\end{proposition}
\begin{proof}
The space $P(X)$ of Borel probability measures on $X$ is weak-* compact and metrizable and the operator $P_{\mu}(\nu)=\mu \star \nu$ acts by isometries. The Banach fixed point theorem guarantees a fixed point, which is by definition a stationary measure. More explicitly, for any $x\in X$,
any weak limit of a subsequence of the probability measures
$$\nu_{n}=\frac{1}{n}\sum^{n-1}_{k=0}\mu^{*k}*\delta_{x}$$ on $X$ is clearly stationary, and such a limit exists by compactness of $P(X)$

\end{proof}

The following is due to Furstenberg \cite{Furstenberg}; the proof uses Doob's dominated Martingale convergence theorem.
\begin{proposition}\label{stationaryconvergence}
Let $\nu$ be a $(G,\mu)${-}stationary measure on $B$.
For $P$ almost every sample path $\omega \in G^\N$, $\omega_{n}\nu$ weakly converges to some measure $\nu_{\omega}$.
Moreover, $\nu$ decomposes as
$$\nu = \int_{\omega \in G^\N}\nu_{\omega} dP(\omega)$$
\end{proposition}

The measure $\nu$ on $B$ is said to be preserved by the action $G\curvearrowright B$ if for every $g\in G$ and any Borel $A\subset B$ we have $\nu(A)=0$ if and only if $\nu(gA)>0$.

The following is classical, see e.g. the survey by Furman \cite{Furman}. We provide a proof for completeness.
\begin{theorem} \label{basicstationary}
Let $G\curvearrowright B$ be any action of a countable group by homeorphisms  on a compact Hausdorff space. Let $\mu$ be a measure on $G$ and $\nu$ a $(G,\mu)${-}stationary measure on $B$. If every orbit of $G$ is infinite and the support of $\mu$ generates $G$ as a semi-group, then $\nu$ has no atoms and it's measure class is preserved by the action $G\curvearrowright X$. If in addition the action is minimal, then $\nu$ has full support on $X$.
\end{theorem}
\begin{proof}
Let $x\in X$.
Any weak limit of a subsequence of the probability measures
$$\nu_{n}=\frac{1}{n}\sum^{n-1}_{k=0}\mu^{*k}*\delta_{x}$$ on $X$ is clearly stationary, and such a limit exists
since the space of probability measures on $X$ with the weak topology is compact.
Assume every $G$ orbit in $B$ is infinite.
We {first} show that $\nu$ has no atoms.
Suppose $\nu$ has an atom. Since $\nu$ is finite, there must be an atom $b\in X$ of maximal mass. Then $\nu(gb)\leq \nu(b)$ for all $g\in G$.
By stationarity we have for each $n>0$
$$\nu(b)=\sum_{g\in G}\mu^{*n}(g)\nu(g^{-1}b).$$
Since $\mu^{*n}$ is a probability measure, and $\nu(g^{-1}b)\leq \nu(b)$ for all $g$ it follows that $\nu(g^{-1}b)=\nu(b)$ for all $g$ with $\mu^{*n}(g)>0$. Since the support of $\mu$ generates $G$, for all $g\in G$ there is an $n>0$ with $\mu^{*n}(g)>0$ so $\nu(g^{-1}b)=\nu(b)$ for all $g\in G$. Since the orbit $Gb\subset B$ is infinite, this contradicts finiteness of $\nu$.

{Now, we show that $G\curvearrowright X$ preserves the measure class of $\nu$. Indeed, suppose $\nu(A)=0$.
By stationarity we have for each $n>0$
$$0=\nu(A)=\sum_{g\in G}\mu^{*n}(g)\nu(g^{-1}A).$$
Since $\nu(g^{-1}A)\geq 0$ it follows that $\nu(g^{-1}A)=0$ for all $g$ with $\mu^{*n}(g)>0$. Since the support of $\mu$ generates $G$, for all $g\in G$ there is an $n>0$ with $\mu^{*n}(g)>0$ so $\mu(g^{-1}A)=0$ for all $g\in G$.}

Finally, assume $G\curvearrowright B$ is minimal.
Suppose $A\subset B$ is an open set.
If $\nu(A)=0$ then quasi-invariance implies $\nu(gA)=0$ for all $g\in G$.
But minimality of the action $G\curvearrowright B$ implies $B=\cup_{g\in G} gA$ which would mean $\nu(B)=0$ contradicting the fact that $\nu$ is a probability measure. Thus $\nu$ has full support on $B$.
\end{proof}
\begin{definition} \cite[Definition 2.10]{Furman}
Let $G\curvearrowright B$ be a  minimal action by homeomorphisms with every orbit infinite. Let $\mu$ be a measure on $G$ whose support generates $G$ as a semi-group. Let $\nu$ be a $\mu$ stationary measure on $B$.
The pair $(B,\nu)$ is called a $(G,\mu)$ boundary if for $P$ almost every sample path $\omega\in G^\N$, there is an $\omega_{+}\in B$ such that  $\omega_{n}\nu$ converges to the Dirac measure $\delta_{\omega_{+}}$ at $\omega_{+}$.
\end{definition}
Suppose $(B,\nu)$ is a $(G,\mu)$ boundary.
Then we have a measurable map \cite{Furman}
$bnd:G^\N\to B$, $bnd(\omega)=\omega_{+}$ and the disintegration formula from Proposition \ref{stationaryconvergence} implies
$$\nu(A)=P(\omega\in G^\N: \omega_{+}\in A)$$ for every Borel $A\subset B$.
In other words, $$\nu =bnd_{*}P.$$
Let $\hat{\nu}=bnd_{*}\hat{P}$ be the similarly constructed stationary measure for $\hat{\mu}$

Consequently, for {$\overline{P}$-almost} every $\omega \in G^\Z$
there are distinct (as $\nu$ and $\hat{\nu}$ are non-atomic) points $bnd_{\pm}(\omega)=\omega_{\pm}\in B$ such that
$$\omega_{n}\nu\to \delta_{\omega_{+}}$$ and
$$\omega_{-n}\hat{\nu}\to \delta_{\omega_{-}}$$ as
$n\to +\infty$
and
$$\nu\times \hat{\nu}=(bnd_{+}\times bnd_{-})_{*}\overline{P}.$$

The following is proved by Kaimanovich in Theorem 6.3 of
\cite{KaiStrip} (the assumptions are different but the proof carries over verbatim)
\begin{proposition}\label{stationaryergodic} \cite[Theorem 6.3]{KaiStrip}
Let $(B,\nu)$ be a $(G,\mu)$ boundary.
The (diagonal) $G$ action on $B\times B$ is $\nu \times \hat{\nu}$ ergodic.
\end{proposition}
\begin{proof}
Note, $$bnd_{\pm}U^{k}\omega= \omega^{-1}_{k}bnd_{\pm}\omega$$
Thus, if $A\subset B\times B$ is $G${-}invariant with $0<(\nu\times \hat{\nu})(A)<1$ then
$(bnd_{+}\times bnd_{-})^{-1}(A)$ is $U${-} invariant with $0<\overline{P}((bnd_{+}\times bnd_{-})^{-1}(A))<1$.
Since $U$ is ergodic with respect to $\overline{P}$ we get the result.
\end{proof}
\begin{corollary}\label{stationaryunique}
If $(B,\nu)$ is a $(G,\mu)$ boundary,
then $\nu$ is the unique $\mu$ stationary measure on $B$.
\end{corollary}
\begin{proof}
By Proposition \ref{stationaryergodic}, any $\mu${-}stationary measure on $B$ is ergodic with respect to  {the action} $G\curvearrowright B$. On the other hand, by the ergodic decomposition for stationary measures (see e.g. \cite{Furman}) any two distinct ergodic stationary measures are singular and their average is a nonergodic stationary measure.
\end{proof}

\begin{theorem}\label{convergencegroup}
Let $G\curvearrowright B$ be a minimal convergence action of a countable group $G$ on an infinite compact Hausdorff space $B$. Let $\mu$ be a measure on $G$ such that the support of $\mu$ generates $G$ as a semigroup. Let $\nu$ be a $(G,\mu)${-}stationary measure on $B$. Then $(B,\nu)$ is a $(G,\mu)$ boundary.
\end{theorem}

This follows by repeating almost verbatim the arguments in Lemma 2.2 and Theorem 2.4 of Kaimanovich in \cite{KaiStrip}. We reproduce the argument for completeness.

\begin{lemma}\label{deltameasure}
Let $a,b,c\in B$ with $a\neq b$ and $g_{n}\in G$ a sequence such that $g_{n}a\to c$, $g_{n}b\to c$.
Let $\kappa$ be any {probability} measure on $B$ with no atoms. Then $g_{n}\kappa\to \delta_{c}$.
\end{lemma}
\begin{proof}
By definition of convergence action, there is a $p\in B$ with $g_{n}q\to c$ for all $q\in B\setminus \{p\}$, uniformly over compact subsets of $B\setminus\{p\}$. Since $\kappa$ does not have an atom at $p$ it follows that $g_{n}\kappa \to \delta_{c}$.
\end{proof}
\begin{proof}[Proof of Theorem \ref{convergencegroup}]
Since $G$ is nonamenable, $P$ almost every $\omega \in G^\Z$ is unbounded in $G$.
Thus, we can find a subsequence $\omega'$ of $\omega$ and points $a,b,c \in B$ with $a\neq b$ such that $\omega'_{n}a\to c$ and
$\omega'_{n}b\to c$.
By Lemma \ref{deltameasure}, since $\nu$ has no atoms, we have $\omega'_{n}\nu \to \delta_{c}$.
But by stationarity of $\nu$ we know $\omega_{n}\nu$ converges to some measure $\kappa$ so $\kappa=\delta_{c}$ whence
$\omega_{n}\nu\to \delta_{c}$.
\end{proof}
From Theorem \ref{convergencegroup} and Lemma \ref{stationaryunique} we immediately obtain the following.
\begin{corollary}\label{uniqueconvergence}
Let $G\curvearrowright Z$ be a minimal convergence action of a countable group $G$ on an infinite compact Hausdorff space $Z$. Let $\mu$ be a probability measure on $G$ such that the support of $\mu$ generates $G$ as a semigroup. Then there is a unique $\mu$-stationary measure $\nu$ on $Z$.
\end{corollary}
In the proofs of Proposition \ref{borel} and Theorems \ref{conicalgeneric} and \ref{conicalprecise} the following characterizations of conical points for convergence actions, which follow easily from the definition.
\begin{lemma}\label{conicalconvergence}
A point $x\in B$ is conical if and only if there are distinct $y,z\in B\setminus\{x\}$, a neighborhood $E$ of the diagonal in $B\times B$ and a sequence $g_{n}\in G$ converging to $x$ such that we have ]$(g^{-1}_{n}x,g^{-1}_{n}z)\notin E$ for infinitely many $n$ and $(g^{-1}_{n}x,g^{-1}_{n}y)\notin E$ for infinitely many $n$.
 \end{lemma}

\begin{lemma}\label{conicalpairs}
Let $G\curvearrowright B$ be a minimal convergence action.
Suppose $x,y\in B$ are distinct points. The following are equivalent.

a) At least one of $x,y$ is conical.

b) There is some infinite sequence $g_{n}$ in $G$ and a neighborhood $E$ of the diagonal in $B\times B$ such that $(g_{n}x,g_{n}y)\notin E$ for any $n$.
\end{lemma}
%{We will assume that $X$ an infinite compactum with a countable base (i.e. Hausdorff, compact, second countable), in particular a metrizable space}.%

We now prove:
\begin{proposition}\label{borel} Let {$X$ be a compact second countable Hausdorff space. Let $G\curvearrowright X$ be a nonelementary convergence action of a countable group $G$ on $X$.
Then the {conical limit points} of the action  form a Borel subset of $X$.}
\end{proposition}
\begin{proof}
{Let   $B\subset X$ denote the limit set of the action.}
For $E\subset B\times B$ let $A_{n}(E)\subset B\times B$ consist of pairs $x,y$ such that there are at least $n$ distinct $g\in G$ with $(gx,gy)\notin E$.

Then $A_{n}(E)$ can be written as the union of $\cap^{n}_{i=1}g^{-1}_{i}E^{c}$ over all distinct $n$-tuples $(g_{1},..,g_{n})$ of $G$ so it is Borel as long as $E$ is Borel.

Let $A(E)=\cap_{n\in \mathbb{N}}A_{n}(E)$, which is also Borel whenever $E$ is Borel.

Let $A=\cup_{n\in \mathbb{N}}A(E_{n})$ where $E_{n}$ is a decreasing sequence of neighborhoods of the diagonal in $B\times B$ whose intersection is the diagonal (this is where we use the second countability axiom.

Then $A$ is a Borel subset of $B\times B$.

By Lemma \ref{conicalpairs} $A$ exactly consists of the pairs $(x,y)$ with at least one of $x,y$ conical.

The complement ${W}=(B\times B)\setminus A$, which consists precisely of pairs of points neither of which is conical, is thus a Borel subset of $B\times B$. {Since the diagonal $\Delta(X)=\{(x,x) : x\in X\}$ is a Borel subset of $X$ we obtain that $Z=W\cap \Delta(X)=\{(z,z),\ z\ \rm{ is\ not\ conical}\}$ is also Borel.
The subset $NC\subset B$ consisting of non-conical points of $B$ coincides with $f^{-1}(Z)$ for  the continuous map $f:B\to B\times B$,
$f(x)=(x,x)$.} Hence $NC$ is Borel.

Thus, the set of conical points, which coincides with $B\setminus NC$, is also Borel.
\end{proof}

Consider a convergence action of a countable group $G$ on   $X$. {We say that} a sequence $g_{n}\in G$ converges to $b\in B=\Lambda G$ if   $g_{n}x\to b$ for all but at most one $x\in B$, {note that it is enough to request it for two distinct points}.

{If} $\nu$ is any non-atomic measure on $B$ then $g_{n}\to b$ if and only if $g_{n}\nu\to \delta_{b}$.
Thus we see that for {$\overline{P}$-almost} every $\omega \in G^{\mathbb{Z}}$, there are $\omega_{\pm}\in B$ with   $\omega_{\pm n}\to \omega_{\pm}$ as $n\to \infty$.

We want to prove:
\begin{theorem}\label{conicalgeneric}
Let $G\curvearrowright X$ be a nonelementary convergence action of a countable group. Let $\mu$ be a measure on $G$ whose support generates $G$ as a semigroup and $\nu$ the $(G,\mu)${-}stationary measure on the limit set $B\subset X$.
Then the non-conical points of $X$ are contained in a set of $\nu$ measure zero.
\end{theorem}
We deduce Theorem \ref{conicalgeneric} from the following statement.

\begin{theorem}\label{conicalprecise}
For every $c\in (0,1)$ there is a neighborhood $E$ of the diagonal in $B\times B$ with the following property.
For {$\overline{P}$-almost} every $\omega \in G^\Z$
$$\lim \inf_{N\to \infty}\frac{|\{n\in [0,N]:(\omega^{-1}_{n}\omega_{-},\omega^{-1}_{n}\omega_{+})\notin E\}|}{N}>c$$
In particular, {for {$\overline{P}$-almost} every $\omega \in G^\Z$}
there are infinitely many $n\in \Z$ such that $(\omega^{-1}_{n}\omega_{-},\omega^{-1}_{n}\omega_{+})\notin E$
\end{theorem}
\begin{proof}[Proof of Theorem \ref{conicalgeneric} assuming Theorem \ref{conicalprecise}]
Recall {that} $\hat{\nu}\times \nu$ is the pushforward of $\overline{P}$ under the $\overline{P}$ {measurable} map $(bnd_{-}\times bnd_{+}):G^{\mathbb{Z}}\to B\times B$.
Furthermore, the definition $\overline{P}=T_{*}\mu^{\mathbb{Z}}$ implies that $(\omega_{n})_{n>0}$ and $(\omega_{n})_{n<0}$ are independent as $P$ measurable random variables, {in particular}  $(\omega_{n})_{n>0}$ is independent of $\omega_{-}\in B$.
Thus, from the {last claim} of Theorem \ref{conicalprecise} we obtain a neighborhood $E$ of the diagonal in $B\times B$ satisfying the following.
For {$\nu$-almost} every $x=\omega_{+}\in B$ there is a sequence $g_{n}=\omega_{n}$ converging to $x$ as $n\to \infty$ such that for $\hat{\nu}$ almost every $y=\omega_{-}\in B$ $(g^{-1}_{n}{y},g^{-1}_{n}x)\notin E$ for infinitely many $n\in \mathbb{N}$. (the fact that $(\omega_{n})_{n>0}$ is independent of $\omega_{-}$ allows us to choose the same sequence $g_{n}$ independent of $y$).
{By Theorem \ref{basicstationary}} $\hat{\nu}$ is nonatomic, {so for
$\nu$-}almost every $x\in B$ there is a sequence $g_{n}\in G$ converging to $x$ and two distinct $y,z\in B$ such that {each of} $(g^{-1}_{n}y,g^{-1}_{n}x)\notin E$  $(g^{-1}_{n}z,g^{-1}_{n}x)\notin E$ {holds} {for infinitely many $n$}.
By Lemma \ref{conicalconvergence} any such $x$ is conical.
\end{proof}
 We now prove Theorem \ref{conicalprecise}.
 \begin{proof}[Proof of Theorem \ref{conicalprecise}]
By passing to the limit set, we can assume {that} the action is minimal {and so $B=X$}. Since $\nu$ and $\hat{\nu}$ have no atoms we have $(\nu\times \hat{\nu})({\Delta(X))=0}$ {where $\Delta(X)$ is the diagonal $\{(x,x)\ :\ x\in X\}.$}
Thus, there is a neighborhood $E$ of the diagonal such that $(\nu\times \hat{\nu})(E)<1-c$.
Since $\nu=bnd^{*}P$ this implies that the $\overline{P}$ measurable set $\Omega$ of $\omega \in G^\Z$ such that $(\omega_{+},\omega_{-})\in E$ has $\overline{P}$ measure less than $1-c$.

Note, $U^{n}\omega \in \Omega$ if and only if $(\omega^{-1}_{n}\omega_{+},\omega^{-1}_{n}\omega_{-})\in E$.
Recall, $U:G^\Z\to G^\Z$ is measure preserving and ergodic with respect to $\overline{P}$.
Therefore, by the Birkhoff ergodic theorem {for {$\overline{P}$-almost} every $\omega \in G^{\mathbb{Z}}$} we have

$$\lim \inf_{N\to \infty}\frac{|\{n\in [0,N]:U^{n}\omega\in \Omega\}|}{N}\to \overline{P}(\omega)<{1-c}$$
as $N\to \infty$ in other words

$$\lim \inf_{N\to \infty}\frac{|\{n\in [o,N]:(\omega^{-1}_{n}\omega^{-1}_{-},\omega^{-1}_{n}\omega_{+})\notin E\}|}{N}>c$$
\end{proof}

\begin{remark}
If the action $G\curvearrowright B$ is geometrically finite, then there are only countably many non-conical points in $\Lambda G$ {\cite[Main Theorem, a)]{Ge1}}. Thus, Proposition \ref{borel} is immediate and Theorem \ref{conicalgeneric} follows {in this case} from the fact that $\nu$ has no atoms.
\end{remark}
Assume now that $G$ acts by isometries on a metric space $(X, d_X)$
and let $x_0 \in X$.
If $\mu$ has finite first moment (i.e. $\sum_{g\in G}d(g x_{0},x_{0})\mu(g)<\infty$) Kingman's subadditive ergodic theorem
implies that for $P$ a.e. sample path $\omega$
the limit
$$L=\lim_{n\to \infty}\frac{d(\omega_{n}x_{0},x_{0})}{n}$$ exists.
This number $L$ is called the
\emph{drift} of the random walk
induced by $\mu$ with respect to the metric $d_X$.

%If $(X,d_X)$ is a separable geodesic Gromov hyperbolic metric space.
%the action of $G$ on X is called \emph{nonelementary} if $G$ contains a
%pair of loxodromic isometries with disjoint sets of fixed points in
%the \emph{Gromov boundary} $\partial X$ of $X$.
A probability
measure on $G$ is called \emph{nonelementary}
if the subgroup of G generated by its
support is a nonelementary subgroup of $G$.
The following results are due in this generality to Maher and Tiozzo \cite{Maher-Tiozzo}; in the proper setting they were earlier proved by Kaimanovich in \cite{KaiStrip}.

\begin{theorem}\label{Maher-Tiozzo}
 Let $G$ be a countable group that acts by isometries on a
separable {geodesic} Gromov hyperbolic space {$(X, d_X).$ }
%such that any two points in $X\cup \partial X$ can be connected by a geodesic.
Let $\mu$ be a nonelementary probability measure
on $G$. Then for any $x\in X$ and $P$ a.e.
sample path $\omega=(\omega_{n})_{n\in \mathbb{N}}$ of the random walk
on $(G, \mu)$, the sequence $(\omega_{n}x_0)_{n\in \mathbb{N}}$ converges
to a point $\omega_{+}=bnd_{*}\omega \in \partial X$.

If in addition $\mu$ has finite first moment with respect to the metric
$d_X$, then there exists $L_{X}>0$ such that
for $P$-a.e. sample path $\omega$ {and for every $x_0\in X$} one has
$$\lim_{n\to \infty}\frac{d_{X}(x_{0},\omega_{n}x_{0})}{n}=L_{X}$$
The measure $\nu=bnd_{*}P$ is the unique $G$ stationary measure on $\partial X$.
\end{theorem}

To prove finiteness of the harmonic invariant measures constructed in  Section 10, we will need the following result.

\begin{proposition}\label{quantrec}
Let $X$ be a {proper $CAT(-1)$ space.}
Let $G<Isom(X)$ be a nonelementary subgroup of isometries. Let $\mu$ be a measure on $G$ whose finite support generates $G$ as a group. Let $\nu$ be the stationary measure on the ideal boundary $\partial X$ {and a basepoint $o\in X$.}
Then for any $c<1$ there is  $R>0$ such that for every $x\in X$ and {$\nu$-almost} every $\alpha\in \partial X$, {one has}
$${\lim \sup_{T\to \infty}}\frac{|\{t\in [0,T]:d( \gamma_{x,\alpha}(t),Go)>R\}|}{T}<c$$

 where $\gamma_{x,\alpha}$ is the unit speed geodesic from $x$ in direction $\alpha$.
\end{proposition}

For $\omega \in G^{\Z}$ such that $\omega_{n}$ converges to distinct points $\omega_{\pm}$
as $n\to \pm \infty$ let $ \gamma_{\omega}$ be the unit speed geodesic from $\omega_{-}$ to $\omega_{+}$ parametrized so that
$ \gamma_{\omega}(0)$ is at minimal distance to $o$.

Since geodesic rays with the same endpoint in $\partial X$ are asymptotic, and using the fact that $\nu=bnd_{*}P$, Proposition \ref{quantrec} follows from the following.

\begin{proposition}
For every $c>0$ there is an $R>0$ such that for {$\overline{P}$-almost} every
$\omega \in G^\Z$,
$$\lim \sup_{T\to \infty}\frac{|\{t\in [0,T]:d( \gamma_{\omega}(t),Go)>R\}|}{T}<c.$$
\end{proposition}
\begin{proof}
Let $d=\max{d(go,o):g\in {\rm supp}\  \mu}$.
Then for every $\omega \in G^\Z$ and every $n$ we have
$d(\omega_{n}o,\omega_{n+1}o)\leq  d$.
Let $\Omega_{0}\subset G^{\Z}$ be the set of bi-infinite sample paths $\omega$ such
that $\omega_{n}$ converges  to distinct points $\omega_{\pm}\in \partial X$
as $n\to \pm \infty$ and such that
$d(\omega_{n}o,o)/|n|\to L$ as $n\to \pm \infty$.
{Since $\nu$ and $\hat{\nu}$ are nonatomic we know $(\nu\times \hat{\nu})(\Delta(\partial X))=0$}, and therefore $\Omega_{0}\subset G^\Z$ has full $\overline{P}$ measure.

For $\omega \in \Omega_{0}$ set ${s_n}=d(\omega_{n}o, {\gamma}(0))$.

{Note $s_{n}\to \infty$ as $n\to \infty$ and (by the triangle inequality) $|s_{n+1}-s_{n}|<d$ for all $n$.} {The points $\gamma(s_n)$ form $d$-net of the curve $\gamma$: for every $t> s_0=d(o,\gamma(0))$ there exists  $s_{n_t}$ such that $d(\gamma(t), \gamma(s_{n_t}))=|t-s_{n_t}| < d.$  We can also assume that $n_t$ is chosen to minimize the expression $|t-{s_n}|$ bounded by $d$.}

Consider $A>0$.
If $d(\omega_{n_{t}}o, {\gamma})\leq A$ then {we claim that}
$d(\omega_{n_{t}}o, {\gamma}(s_{n_{t}}))\leq 2A$. {Indeed,
there is some $s$ with $d(\omega_{n_{t}}o, {\gamma}(s))\leq A$.
{It follows}
$$|s-s_{n_{t}}|=|d({\gamma}(s),{\gamma}(0))-d(\omega_{n_{t}}o, {\gamma}(0))|\leq
d(\omega_{n_{t}}o, {\gamma}(s))\leq A$$ so
$$d(\omega_{n_{t}}o, {\gamma}(s_{n_{t}}))\leq d(\omega_{n_{t}}o, {\gamma}(s)) + d({\gamma}(s_{n_{t}}),{\gamma}(s))\leq 2A$$} {as was claimed.}

{Since $d(\gamma(s_{n_{t}}), \gamma(t)=\vert s_{n_t}-t\vert\leq d$ by triangle inequality we obtain}
$d(\omega_{n_{t}}o, {\gamma}(t))\leq 2A+d$.

Thus, if $t>d(o, {\gamma})$ is such that $d( {\gamma}(t),Go)>2A+d$ we have $t\in[{s_n}-d,{s_n}+d]$ for an $n$ with
$d(\omega_{n}o, {\gamma})> A$.
{Since $\omega\in \Omega_0$ we have} ${s_n}/n\to L$ as $n\to \infty,$ and hence
$s_{n_{t}}/n_{t}\to L$ as $t\to \infty$.
Since $|t-s_{n_{t}} |<d$ this implies
$t/n_{t}\to L$ as $t\to \infty$.
Thus for each $\omega \in \Omega_{0}$ there is an $T_{0}=T_{0}(\omega)>0$ such that for $t>T_{0}$,  $n_{t}\leq 1.1t/L$.
Hence if $t>T_{0}$ and
$d( {\gamma}(t),Go)>2A+d$ we have
$t\in [{s_n}-d,{s_n}+d]$ for some $n<1.1t/L$ with $d(\omega_{n}o, {\gamma})>A$.
Let $\Omega_{A}\subset \Omega_{0}$ be the measurable set of sequences $\omega$ such that $d(o, \gamma_{\omega})>A.$

The complements $\{\Omega^{c}_{N}\}_{N\in \mathbb{N}}$ form an increasing sequence of sets with union $\Omega_0$.  Clearly each $\omega \in \Omega_{0}$ is in $\Omega^{c}_{N}$ for some natural $N$ so $\cup^{\infty}_{N=1}\Omega^{c}_{N}=\Omega_{0}$ whence
$$\overline{P}(\Omega^{c}_{N})\to \overline{P}(\Omega^{c}_{0})=1$$ and thus $\overline{P}(\Omega_{N})\to 0$ as $N\to \infty$.
Choose $A$ large enough so that
$\overline{P}(\Omega_{A})<\frac{cL}{100d}$.

Note, $U^{n}\omega \in \Omega_{A}$ if and only if
$d(\omega_{n}o, {\gamma_{\omega}})>A$.

By the Birkhoff ergodic theorem, for {$\overline{P}$-almost} every $\omega \in G^\Z$
\begin{equation}\label{birkhoff}\lim_{N\to \infty} \frac{|\{n\in [0,N]: U^{n}\omega \in \Omega_{A}\}|}{N}\to \overline{P}(\Omega_{A})<\frac{c{L}}{100d}\end{equation}
Let $\Omega_{1}\subset \Omega_{0}$ be the full measure set where this convergence holds.

Consider $\omega \in \Omega_{1}$.
Then by (\ref{birkhoff}) there is a $T_{1}(\omega)>T_{0}(\omega)$ such that for $T>T_{1}(\omega)$ the number of integers $n\in [0, {N}]$  {(where $N=[1.1T/L]\to\infty$) satisfying}
$d(\omega_{n}o, {\gamma})>A$ is less than
${N \frac{cL}{100d}}<\frac{cT}{50d}.$
Thus, for $T>T_{1}(\omega)$, the set of $t\in [0,T]$ with $d( {\gamma}(t),Go)>2A+d$ is contained in
$$[0,T_{1}]\cup\bigcup_{\{n\in \N\cap[0,{N}]:d(\omega_{n}o, {\gamma})>A\}}[s_{n}-d,{s_n}+d]$$

\noindent which has length (Lebesgue measure) at most
$T_{1}+2d\frac{cT}{50d}=T_{1}+\frac{cT}{25}$ which is less than $cT$ for large enough $T$.

\end{proof}

\section{Application: harmonic invariant measures}

Let $G$ be a finitely generated nonamenable group and $\mu$ a measure on $ G$ whose support generates $G$ satisfying Assumption 1 and Assumption 2.

{We denote by  $\delta^{f}$} the Floyd metric on the Cayley graph of $ G$ with respect to a finite generating set and a Floyd function $f$. Let $\partial_{f} G$ be the associated Floyd boundary. {We will assume throughout this section that $\partial_{f}G$ has at least three points (and is therefore uncountable).}

Let
$$d_{\mathcal{G}}(x,y)=-\log \frac{\mathcal{G}(x,y)}{\mathcal{G}(e,e)}$$ be the Green metric {and} $\partial_{\mathcal M} G$ be the Martin boundary of $( G,\mu)$.

Let $\Delta:G\times G\times \partial_{\mathcal{M}}G\to \mathbb{R}$ be the Busemann cocycle for $d_{\mathcal{G}}$ constructed in Section 7 and

$$K_{\alpha}(x)=e^{-\Delta(x,o,\alpha)}$$ the Martin kernel.
{Let $\pi:\partial_{\mathcal M}G\to \partial_{f}G$ be the equivariant surjection constructed in Section 7.}

Let $\hat{\mu}(g)=\mu(g^{-1})$ be the reflected measure of $\mu$.  Let $$\hat{\mathcal{G}}(x,y)=\mathcal{G}(y,x)$$ and
$$\hat{d}_{\mathcal{G}}=-\log \hat{\mathcal{G}},$$

be the Green function and Green metric associated to $\hat{\mu}$. Let $\hat{\partial}_{\mathcal M}G$, $\hat{\Delta}$ and $\hat{K}$ be defined by analogy.

Let $\hat{\pi}:\hat{\partial}_{\mathcal M}G\to \partial_{f}G$ be the associated equivariant surjection.

It is known (see e.g. \cite{Sawyer}) that for $P$ almost every $\omega \in  G^{\N}$, $\omega_{n}$ converges to a single point in $\partial_{\mathcal M} G$ and so we can define the harmonic measure $\nu$
on $\partial_{\mathcal M} G$
by

$$\nu(A)=P(\omega\in  G^{\mathbb{N}}:\lim_{n\to \infty}\omega_{n} \in A)$$
From the definition it follows that
$\nu$ is $( G,\mu)$ stationary.
Moreover,
$\nu$ satisfies
\begin{equation}\label{Martinkernel}
\frac{dg\nu}{d\nu}(\alpha)=K(g,\alpha)
\end{equation}
for $\nu$ a.e. $\alpha \in \partial_{\mathcal M} G$ (see e.g \cite[Theorem 5.1]{Sawyer} ).

%(This means $\nu$ is a conformal measure for the action of $ G$ on itself with respect to the Green metric).

Let $\hat{\nu}$ be the similarly constructed $\hat{\mu}$ stationary measure on $\hat{\partial}_{\mathcal M}G$

%which satisfies

%\begin{equation}\label{Martinkernel}\frac{dg\hat{\nu}{d\hat{\nu}(\alpha)=\hat{K}(g,\alpha) \end{equation} for $\hat{\nu}$ a.e. $\alpha \in \hat{\partial}_{\mathcal M} G$

 Karlsson \cite{Karlpoisson} proved that $G\curvearrowright \partial_{f} G$ is a minimal convergence action, and thus (as Karlsson also proved)  by Corollary \ref{uniqueconvergence} there is a unique $\mu$ stationary probability measure on
$\partial_{f} G$, denoted by $\nu_f$.

Since the pushforward $\pi_{\star}\nu$ is stationary, it follows that $\nu_{f}=\pi_{\star}\nu$.

Similarly, $\hat{\nu}_{f}=\hat{\pi}_{\star}\hat{\nu}$ is the unique $\hat{\mu}$ stationary probability measure on $\partial_{f} G$.
By Proposition \ref{stationaryergodic}, $\hat{\nu}_{f}\times \nu_{f}$ is ergodic with respect to the $ G$ action.

By Theorem \ref{basicstationary} $\nu_{f}$ and $\hat{\nu}_{f}$ have no atoms so
$\hat{\nu}_{f}\times \nu_{f}$ assigns zero weight to the diagonal.

For $\beta \in \partial_{\mathcal M} G$, $\alpha \in  \hat{\partial}_{\mathcal M} G$
and
$g\in G$ let

$$\Theta(\alpha,\beta)=\lim \inf_{x\to \alpha}\frac{K(x,\beta)}{\mathcal{G}(x,e)}$$
$$=\lim \inf_{x\to \alpha} \lim_{y\to \beta}\frac{\mathcal{G}(x,y)}{\mathcal{G}(x,e)\mathcal{G}(e,y)}$$
and let
\begin{equation}\label{Gromovgreen}
\rho^{\mathcal{G}}_{e}(\alpha,\beta)=\frac{1}{2}\log \Theta(\alpha,\beta)
\end{equation}

The quantity $\rho^{\mathcal{G}}_{e}(\alpha,\beta)$ can be thought of as an analogue of the Gromov product for the (asymmetric) Green metric (see section 11).

We will prove:
\begin{theorem} \label{invmeasfloyd}
There exists a $ G${-}invariant Radon measure on $\partial_{f} G\times \partial_{f} G\setminus \Delta(\partial_{f}G)$ in the measure class of $\hat{\nu}_{f} \times \nu_{f}$.
\end{theorem}
\begin{proof}
Define a measure on $\hat{\partial}_{\mathcal M} G \times \partial_{\mathcal M} G$ by
$$dm(\alpha,\beta)=\Theta(\alpha,\beta)d\hat{\nu}(\alpha)d\nu(\beta)$$
First we show $m$ is $ G${-}invariant.

Indeed, it is easy to see that
\begin{equation}\label{Naimmult}
\Theta(g^{-1}\alpha,g^{-1}\beta)=
\frac{\Theta(\alpha,\beta)}{\hat{K}(g,\alpha)K(g,\beta)}
\end{equation}
On the other hand
$$\frac{dg\nu}{d\nu}(\beta)=K(g,\beta)$$

and
$$\frac{dg\hat{\nu}}{d\hat{\nu}}(\alpha)=\hat{K}(g,\alpha).$$

Thus, $m$ is $ G${-}invariant.

Let $m_{f}=(\hat{\pi} \times \pi)_{*}m$.

Since $\nu_{f}=\pi_{*}\nu$ we have that $m_{f}$ is a $ G${-}invariant measure on $\partial_{f} G\times \partial_{f} G$ in the measure class of $\hat{\nu}_{f}\times \nu_{f}$.

We need to show $m_{f}$ is locally finite on $\partial_{f} G\times \partial_{f} G\setminus \Delta(\partial_{f}G)$.

It suffices to show that for any disjoint closed subsets $A,B\subset \partial_{f} G$ we have $m_{f}(A\times B)<\infty$ or equivalently $m(\hat{\pi}^{-1}(A)\times \pi^{-1}(B))<\infty$.

For this it is enough to show that $\Theta(\alpha,\beta)$ is bounded over $(\alpha,\beta)\in \hat{\pi}^{-1}(A)\times \pi^{-1}(B)$.

Indeed, there is a $d>0$ and disjoint neighborhoods $U$ and $V$ of $A$ and $B$ in $G\cup \partial_{f} G$ such that    $\delta^{o}_{f}(a,b)>d$ for all $a\in A$, $b\in B$.

Thus, by  Theorem \ref{relancona}
$$\frac{\mathcal{G}(x,y)}{\mathcal{G}(o,x)\mathcal{G}(o,y)}\leq R(d)$$ for all $x\in U\cap G$, $y\in V\cap G$ and hence by taking limits,
$\Theta(\alpha,\beta)\leq R(d)$ for all $\alpha \in \pi^{-1}U$, $\beta \in \pi^{-1}V$ completing the proof.
\end{proof}

{\begin{remark}
Taking logarithms and rearranging in (\ref{Naimmult}) we obtain
\begin{equation}\label{Naimadd}
2\rho^{\mathcal{G}}_e(g^{-1}\alpha, g^{-1}\beta)-2\rho^{\mathcal{G}}_e(\alpha, \beta)=\hat{\Delta}(g,e,\alpha)+\Delta(g,e,\beta)
\end{equation} which is reminiscent of the usual relation between the Gromov product and Busemann cocycle. %(see section 10).
\end{remark}}

We will now consider an action of $G$ on a proper Gromov hyperbolic space $X$, such that the action $G\curvearrowright \partial X$ is nonelementary and geometrically finite.
{Recall that} {in this case we call the action $G\curvearrowright X$ geometrically finite as well.}

By work of Bowditch \cite{Bowditch}, the existence of such an action requires $G$ to be relatively hyperbolic.
	
{The action $G\act X$ is a  convergence action, hence by Corollary \ref{uniqueconvergence}, there is a unique $\mu$ stationary measure $\nu_{X}$ on the limit set $\Lambda(G)\subset \partial X$.}

We will denote by $\nu_{X}$ (resp. $\hat{\nu}_{X}$) the unique $\mu$ (resp. $\hat{\mu}$) stationary probability measure on  $\partial X$.

By Proposition \ref{stationaryergodic} $\hat{\nu}_{X}\times \nu_{X}$ is ergodic with respect to the $ G$ action, and by Theorem \ref{basicstationary} $\nu_{X}$ and $\hat{\nu}_{X}$  have no atoms. Thus,
$\hat{\nu}_{X}\times \nu_{X}$ gives zero weight to the diagonal.

%Let $$\pi:\partial_{\mathcal M} G\to \partial_{f} G$$ be the extension of the identity map on $ G$ constructed in Section 6.

Let $$F:\partial_{f} G\to \Lambda(G)\subset \partial X$$ be the continuous equivariant map on $G$ obtained by Gerasimov in \cite{Gerasimov}.

Let $\varphi=F\circ \pi:\partial_{\mathcal M} G\to \partial X$, and

$\hat{\varphi}=F\circ \hat{\pi}:\hat{\partial}_{\mathcal M} G\to \partial X$

By uniqueness of stationary measures on $\partial X$ and the fact that the pushforward $\pi_{*}\nu$ is stationary we have:  $$\nu_{X}=F_{\star}\nu_{f}=\varphi_{\star}\nu$$ and similarly

$$\hat{\nu}_{X}=F_{\star}\hat{\nu}_{f}=\varphi_{\star}\hat{\nu}$$
%Note, $\varphi^{-1}(p)$ is a singleton for every conical $p\in \partial X$, hence for all but countably many $p\in \partial X$.

%Since $X$ is uncountable and $\nu_{X}$ has no atoms, it follows that $\varphi^{-1}(p)$ is a singleton for $\nu_{X}$ a.e. $p\in \partial X$.

%Thus,  ${\varphi}$ is a measurable isomorphism $(\partial_{\mathcal M}G,\nu)\to (\partial X, \nu_{X})$.

%Define a $G$ equivariant map $\psi:\partial X\to \partial_{\mathcal M}G$ by
%$\psi(\varphi(a))=a$ for $\nu$ almost every $a$
%and $\varphi(\psi(\alpha))=a$ for $\nu_X$ a.e. $\alpha$.

%For instance, one can define $\psi(a)=\phi^{-1}(a)$ when $a\in \partial X$ is conical, and define $\psi$ to be constant on the complement of conical points.

%By (\ref{Martinkernel}) we have:

%\begin{equation}\label{transferMartinkernel}
%\frac{dg\nu_{X}}{d\nu_{X}}(a)=K(g,\psi(a))
%\end{equation}
%for $\nu_{X}$-a.e. $a\in \partial X$.
%(This also follows from \cite{KaiStrip} where it is observed that both are models for the so-called Poisson boundary of $(G,\mu)$).

The following is a corollary of Proposition \ref{invmeasfloyd}
\begin{corollary} \label{invmeasure}
There exists a $ G${-}invariant Radon measure $m_X$ on $\partial X\times \partial X$ in the measure class of $\hat{\nu}_{X} \times \nu_{X}$.
\end{corollary}
\begin{proof}
Let $m_{f}$ be the measure constructed in Theorem \ref{invmeasfloyd}.

Let {$m_{X}=(F \times F)_{*}m_{f}.$}

By $ G${-}equivariance of ${\varphi}$, and since $\nu_{X}={\varphi*}\nu$ we have that $m_{X}$ is a $G${-}invariant measure on $\partial X\times \partial X$ in the measure class of $\hat{\nu}_{X}\times \nu_{X}$.
If $A,B\subset \partial X$ are disjoint closed sets then $F^{-1}(A)$ and $F^{-1}(B)$ are disjoint closed subsets of $\partial_{f} G$ so {$$m_{X}(A\times B)=m_{f}(F^{-1}(A)\times F^{-1}(B))<\infty.$$}

Thus $m_X$ is Radon.
\end{proof}

\noindent {The main application is the following.}
\begin{theorem} \label{finiteharmonicinvariant}
Assume $X$ is a proper $CAT(-1)$ space {and $G<Isom X$ is a non-elementary subgroup acting geometrically finitely on X. }
Let $\mu$ be a probability measure on $G$ whose finite support generates $G$.
Then the measure $m_{X}$ defined above is the geodesic current for a $G$ and geodesic flow invariant measure $\tilde{L}$ on the unit tangent bundle $T^{1}X=\partial^{2}X\times \mathbb{R}$ projecting to a finite ergodic geodesic flow invariant measure  $L$ on $T^{1}M/ G$.
\end{theorem}

\begin{proof}
Let $M=X/G$ and $T^{1}M=T^{1}X/G$ {be} the unit tangent bundle of $M$. {Denote by}
$\Phi:T^{1}X\to T^{1}M$ and $p:X\to M$   the canonical projections.
Let $\tilde{g_{t}}$ and $g_{t}$ ($t\in \mathbb{R}$)  the geodesic flow on $X$ and $M$ respectively.
Define a measure $\tilde{L}$ on $T^{1}X$ by
$$d\tilde{L}(q)=dm_{X}(q^{-},q^{+})dt$$ where $dt$ is geodesic arclength. Since $m_{X}$ is locally finite, so is $\tilde{L}$. By definition, $\tilde{L}$ is $\tilde{g_{t}}${-}invariant. Furthermore, the $G${-}invariance of $m_{X}$ implies the $G$ invariance of $\tilde{L}$.
Thus, $\tilde{L}$ projects to a $g_{t}${-}invariant measure $L$ on $T^{1}M$.
Since $m_X$ is ergodic with respect to the action of $G$ on $\partial^{2}X=\mathbb{R} \ T^{1}X$, $L$ is ergodic with respect to the geodesic flow on $T^{1}M=T^{1}X/G$.

We want to show {that} $L$ is finite.
Indeed, suppose $L$ is infinite. Then Hopf's infinite ergodic theorem (see e.g. \cite{Hochman}) implies that for every  $A\subset T^{1}M$ with $L(A)<\infty$ and for $L$ almost every $q\in T^{1}M$:
$$\frac{|\{t\in [0,T]:g_{t}q\notin A\}|}{T}\to 1$$ as $t \to \infty$.

Equivalently: for any $C\subset T^{1}X$ with $L(\Phi(A))<\infty$, for $\hat{\nu}_{X}\times \nu_{X}$ almost every $(q^{+},q^{-})\in \partial^{2}X$, and for any $q\in T^{1}X$ with endpoints $q_{\pm}$, we have
$$\frac{|\{t\in [0,T]:\tilde{g_{t}}q\notin C\}|}{T}\to 1$$ as $T\to \infty$.

Taking $C=G\dot T^{1}B_{R}(o)\subset T^{1}X$ we see that $\Phi(C)=T^{1}B_{R}(p(o))$ is compact so $L(\Phi(C))<\infty$ as $L$ is locally finite.

Thus, for any $R>0$, for $\hat{\nu}_{X}\times \nu_{X}$-a.e. $(q^{+},q^{-})\in \partial^{2}X$, and for any unit speed geodesic $\gamma$ connecting $q_{\pm}$ we obtain $$\frac{|\{t\in [0,T]:d(\gamma(t),Go)>R\}|}{T}\to 1$$ as $T\to \infty$.
This contradicts Proposition \ref{quantrec}.

\end{proof}

\begin{remark}
The finite support assumption in Theorem \ref{finiteharmonicinvariant}
is needed only to apply Proposition \ref{quantrec}.
It seems likely that the latter holds under the weaker assumption of superexponential moment but we could not verify it.
\end{remark}

\begin{remark}
By a general result of Babillot \cite{Babillot} about product measures, the measure $L$ {in}   Theorem \ref{finiteharmonicinvariant} is mixing with respect to the geodesic flow on $T^{1}X$ unless the logarithms of the lengths of closed geodesics on $X/ G$ are contained in a discrete subgroup of $\mathbb{R}$. This mixing condition is satisfied for instance whenever $ G$ contains a parabolic, whenever $X$ is a surface, whenever $X$ is a rank 1 symmetric space, or when $X/ G$ has finite Riemannian volume \cite{Dalbo}.  %It would be interesting to investigate equidistribution properties of random walk trajectories with respect to this measure.
\end{remark}

\section{Application: Singularity of Measures}
Let $(X,d_{X})$ be a proper geodesic Gromov hyperbolic space.

For $x,y,z\in X$ and $a,b\in \partial X$ let
$$\beta^{X}_{z}(x,y)=d_{X}(x,z)-d_{X}({y},z)$$
$$\beta^{X}_{a}(x,y)=\lim \inf_{z\in G, z\to a}\beta^{X}_{z}(x,y)$$

$$\rho^{X}_{z}(x,y)=\frac{1}{2} d_{X}(z,x)+d_{X}(z,y)-d_{X}(y,x)$$

$$\rho^{X}_{z}(a,b)=\lim \inf_{x,y\in X,x\to \alpha, y\to \beta}\rho^{X}_{z}(x,y)$$

The quantities $\rho^{X}$ and $\beta^{X}$ are called the Gromov product and the Busemann cocycle for $(X,d_{X})$ respectively.

Fix a basepoint $o\in X$.

Suppose $G\curvearrowright X$ is a properly discontinuous isometric action by a countable group.

Note, this is necessarily a convergence action.

Let $\Lambda G\subset \partial X$ be the limit set of $G$, and $\Lambda_{c}G\subset \Lambda G$ the $G$ equivariant set of conical limit points.
%An isometry $g\in G$ is called hyperbolic if it has two distinct fixed points on $\partial X$. It is parabolic if it has a single fixed point on $\partial X$, and elliptic if it has a fixed point in $X$, or equivalently if it has finite order, i.e. $g^{n}=id$ for some $n>0$. Every $g\in G$ is exactly one of hyperbolic, parabolic, or elliptic. An isometry $g\in G$ is hyperbolic if and only if $\lim_{n\to \infty} \inf d_{X}(g^{n}o,o)/n>0$. See \cite[Chapter 9]{CDP} for details. In particular, if $g$ is parabolic then $||g^{n}||/d_{X}(g^{n}o,o)\to \infty$ (where $||.||$ denotes the word norm with respect to some finite generating set).

\begin{definition}
A finite Borel measure $\kappa$ on $\partial X$ is called $ G${-}quasiconformal of dimension $s$ and quasiconformal constant $C$ (or simply $(G,s,C)$ quasiconformal) if the following hold:

a) $\kappa$ is supported on the limit set $\Lambda  G \subset \partial X$ of $ G$.

b) $ G$ preserves the measure class of $\kappa$ and for all $g\in  G$
$$C^{-1}e^{-s\beta^{X}_{\zeta}(go,o)}\leq \frac{dg\kappa}{d\kappa}(\zeta)\leq Ce^{-s\beta^{X}_{\zeta}(go,o)}.$$
\end{definition}
Let $$h=h_{X}( G)=\inf \{s:\sum_{g\in  G}e^{-s d_{X}(x,gx)}<\infty\}$$ be the critical exponent of the action.

Note the Poincar{\'e} series
$\sum_{g\in  G}e^{-s d_{X}(x,gx)}$ converges for $s>h$ and diverges for $s<h$.

{Coornaert proved in \cite[Theorem 5.4]{Coornaert} that whenever $G\curvearrowright X$ is a properly discontinuous isometric action on a proper Gromov hyperbolic space with $h(G)<\infty$ there is a $G$-quasiconformal measure on $X$ of dimension $h(G)$.}

%\begin{definition} The action $ G \curvearrowright X$ is said to be of divergence type if $h=h_{X}(G)<\infty$ and $$\sum_{g\in  G}e^{-h d_{X}(x,gx)}=\infty$$ and of convergence type {if this series converges} {or if $h=\infty$}.
%\end{definition}
The following is due in this generality to Matsuzaki-Yabuki-Jaerisch \cite[Theorem 4.1]{MYJ}.

\begin{proposition} \label{dichotomy}
Suppose $G\curvearrowright X$ is a nonelementary properly discontinuous isometric action on a proper geodesic Gromov hyperbolic space.
Let $\kappa$ be a $G$-quasiconformal measure which gives full measure to the set $\Lambda_{c}G$ of conical limit points in $\partial X$. Then $\kappa$ is ergodic with respect to the action $G \curvearrowright \partial X$.

%If the action $G \curvearrowright X$ is of divergence type, then:

%a)  Quasiconformal probability measures are quasi-unique, i.e. there is a $D>0$ such that if $\kappa$ and $\kappa'$ are $(G,s,C)${-}quasiconformal probability measures on $\partial X$ then
%$$D^{-1}\kappa(A)\leq \kappa'(A)\leq D\kappa(A)$$ for any Borel $A\subset \partial X$ {\cite[Theorem 5.2]{MYJ}}.

%b) Any $ G${-}quasi-conformal measure is ergodic, {and gives full mass to conical limit points {\cite[Theorem 4.1]{MYJ}}.}

%On the other hand, if $ G \curvearrowright X$ is of convergence type, any $ G${-}quasiconformal measure assigns zero weight to the conical limit set {\cite[Corollary 2.13]{MYJ}}.
\end{proposition}

The goal of this section is to prove:
\begin{theorem}\label{singularityofmeasures}
Suppose $ G \curvearrowright \partial X$ is a nonelementary, geometrically finite action {with  parabolic elements}  on a proper geodesic Gromov hyperbolic space $X$. Let $\kappa$ be a $ G${-}quasiconformal measure on $\partial X$.

Let $\mu$ be a symmetric probability measure on $G$ with superexponential moment whose  support generates $G$.

Let $\nu_X$ be the $\mu$ stationary measure on $\partial X$.

Then $\nu_X$ and $\kappa$ are singular.
\end{theorem}
We can decompose $\kappa$ as $\kappa_{1}+\kappa_{2}$ where 
$\kappa_{1}(A)=\kappa(A\cap \Lambda_{c}G$ is the restriction of $\kappa$ to the conical limit set and $\kappa_{2}(A)=\kappa(A\setminus \Lambda_{c}G$ the restriction to its complement in $\Lambda G$.

Then the $\kappa_{i}$ are both $G$-quasiconformal. To prove $\nu_X$ is singular to $\kappa$, it is enough to show it's singular to both $\kappa_1$ and $\kappa_2$.
Note, by definition $\kappa_2$ gives zero weight to conical limit points in   $\partial X$, while by Theorem \ref{conicalgeneric} $\nu_X$ gives full weight to the same set, so $\kappa_2$ is singular to $\nu_X$.

Therefore, we may assume $\kappa=\kappa_1$, i.e. $\kappa$ is supported on $\Lambda_{c}G$.
\begin{lemma}
\label{singularvsbounded}
Either $\kappa$ and $\nu_X$ are singular or there is a constant $C\geq 1$ such that
$$C^{-1}\kappa(A)\leq\nu_{X}(A)\leq C\kappa(A)$$ for all Borel $A\subset \partial X$.
\end{lemma}
\begin{proof}
The action $G\curvearrowright X$ preserves the measure class of $\nu_{X}$ by Theorem \ref{basicstationary} and that of $\kappa$
by the positivity of $\frac{dg\kappa}{d\kappa}$ in the definition of quasiconformality.
Furthermore, $G\curvearrowright X$ is ergodic with respect to $\nu_{X}$ by Proposition \ref{stationaryergodic} and ergodic with respect to $\kappa$ by Proposition \ref{dichotomy}.
Any two measures whose measure classes are preserved by an action $G\curvearrowright X$ and are ergodic with respect to this action are either singular or mutually absolutely continuous, so this is true for $\kappa$ and $\nu_X$. Suppose they are  mutually absolutely continuous. Then there is a positive Borel function $J$ with
$d\kappa=J d\nu_X$.
To prove Lemma \ref{singularvsbounded} it suffices to prove that $J$ is $\nu_{X}$-essentially bounded: that is, there exists a $K>0$ with
$J(a)\in [K^{-1},K]$ for $\nu_X$-a.e. $a\in \partial X.$
By the chain rule we have \begin{equation}\label{RNcomp}
J\circ g^{-1}=\frac{d g\kappa}{dg\nu_{X}}=\frac{d g \kappa}{d\kappa}\frac{d\kappa}{d\nu_{X}}
\frac{d\nu_{X}}{dg\nu_{X}}.
\end{equation}
$\nu_{X}$-a.e.

  {Let $\varphi:\partial_{\mathcal M} G\to\partial X$ be the equivariant continuous map having one-point preimage on every conical point of $\partial X$ (see sections 7, {9}). Let $\psi:\partial X\to \partial_{\mathcal M}G$ be a Borel map with $\varphi(\psi(a))=a$ for all conical $a\in \partial X$ (hence for $\nu_{X}$ a.e. $a\in partial X$, since conical points have full $\nu_{X}$ measure by Theorem \ref{conicalgeneric}).
The quasiconformality of $\kappa$ yields
\begin{equation}\label{quasic}
E^{-1}\exp({-h\beta^{X}_{a}(go,o)}){\leq}\frac{dg\kappa}{d\kappa}(a)\leq E\exp({-h\beta^{X}_{a}(go,o)})
\end{equation}
where $E$ is the quasiconformal constant for $\kappa$.
On the other hand, by (\ref{Martinkernel}) we know
\begin{equation}\label{hypkernel}
\frac{d\nu_{X}(a)}{dg\nu_{X}}=\frac{1}{K(g,\psi(a))}=\exp({\Delta(g,e,\psi(a))}),
\end{equation} for $\nu_{X}$-a.e. {$a\in \partial X$ and $\alpha=\psi(a)\in\partial_{\mathcal M}G.$}
Moreover, by definition {we have}
\begin{equation}\label{RN}
\frac{d\kappa}{d\nu_{X}}=J.
\end{equation}
Thus, by plugging in (\ref{quasic}),(\ref{hypkernel}), and (\ref{RN}) into (\ref{RNcomp}) we obtain
\begin{equation}\label{Jcomp}
E^{-1}J(a)\exp{(\Delta(g,e,\psi(a)))-h\beta^{X}_{a}(go,o)}\leq J(g^{-1}(a))
\end{equation}
$$\leq
E J(a)\exp{(h\beta^{X}_{a}(go,o)-\Delta(g,e,\psi(a)))}$$}
$\nu_{X}$-a.e. $a\in \partial X$.

Define the function $\tilde{J}:\partial^{2}X_c\to \mathbb{R}${, where $\partial^2 X_c=\{(a,b)\ :\ a, b\in \partial_cX\}$ and $\partial_cX$ is the subset of conical points in $\partial X,$ as follows:}
\begin{equation}\label{jconf}
\tilde{J}(a,b)=J(a)J(b)\exp{(2h\rho^{X}_{o}(a,b)-2\rho^{\mathcal G}_{e}(\psi(a),\psi(b)))},
\end{equation}
where $\rho^{\mathcal G}$ was defined in (\ref{Gromovgreen}) in section 10.

{We have},

\begin{equation}\label{Naim}
2\rho^{\mathcal{G}}_e(g^{-1}{\alpha}, g^{-1}{\beta})-2\rho^{\mathcal{G}}_e({\alpha}, {\beta})=\Delta(g,e,{\alpha})+\Delta(g,e,{\beta})
\end{equation}

and

\begin{equation}\label{busemann}
2\rho^{X}_o(g^{-1}a, g^{-1}b)-2\rho^{X}_o(a, b)=\beta^{X}_a(go,o)+\beta^{X}_b(go,o)
\end{equation}
Thus plugging (\ref{Jcomp}), (\ref{Naim}) and (\ref{busemann}) into (\ref{jconf}) we obtain for all $g\in G$
\begin{equation}
E^{-2}\leq \tilde{J}(g^{-1}a,g^{-1}b)\leq E^{2}
\end{equation}

%Note, $\tilde{J}=\frac{d\theta}{dm_{X}}$ and so
%$\tilde{J}\circ g^{-1}=\frac{d g \theta}{d g m_{X}}$. Thus the $ G$ invariance of $m_X$ and the bound  %$D^{-1}\leq \frac{dg\theta}{d\theta}\leq D$ imply that for each $g\in G$:

%$$\frac{\tilde{J}(g^{-1}(a),g^{-1}(b))}{\tilde{J}(a,b)}=\frac{dg\theta/dgm}{d\theta/dm}=\frac{dg\theta/dm}{d\theta/dm}=\frac{dg\theta}{d\theta}\in [D^{-1},D]$$
%Thus for all $g\in G$:
%\begin{equation}
%D^{-1}\leq \frac{\tilde{J}(g(a),g(b))}{\tilde{J}(a,b)}\leq D
%\end{equation}

{Choose} $C>0$ {to} be large enough so that {the set}
$$A=\{(a,b)\in \partial X\times \partial X: C^{-1}<\tilde{J}(a,b)<C\}$$
has positive $\nu_{X}\times \nu_{X}$ measure.
Since $\nu_{X}\times \nu_{X}$ is ergodic   {(by Proposition \ref{stationaryergodic})}, this implies {that the subset $$\Omega=G A$$} has full $\nu_{X}\times \nu_{X}$ measure.
{Then} for  every $(a,b)\in \Omega$ there is  $g\in  G$ with $g(a,b)\in A$.
{So}
\begin{equation}
C^{-1}D^{-1}\leq D^{-1} \tilde{J}(ga,gb)\leq \tilde{J}(a,b)\leq D \tilde{J}(ga,gb)\leq DC,
\end{equation} for {all} pairs $(a,b)\in \Omega$.
Thus, letting $K=CD$ we see that
\begin{equation}\label{boundontilde}
K^{-1}\leq \tilde{J}(a,b)\leq K,\forall {(a,b)\in \Omega.}
\end{equation}
%for all $(a,b)$ in a set $\Omega$ of full $\nu_{X}\times \nu_{X}$ measure.}

We will now show that $J$ is $\nu_{X}$-essentially bounded.
Indeed, (\ref{boundontilde}) implies that
$$
K^{-1} \exp{(2h\rho^{X}_{o}({a},{b})-2\rho^{\mathcal{G}}_{e}(\alpha,\beta))} \leq J(a)J(b)$$ $$
\leq K \exp{(2h\rho^{X}_{o}({a},{b})-2\rho^{\mathcal{G}}_{e}(\alpha,\beta))}$$

\noindent for  all {$(\alpha,\beta) \in \varphi^{-1}(\Omega)$}  {and $a=\varphi(\alpha), b=\varphi(\beta).$}

Let $U,V \subset \partial X$ be disjoint closed sets with nonempty interiors. Then {by Theorem \ref{relancona}} $\rho^{\mathcal{G}}_{e}({\alpha,\beta})$ is bounded over all $(\alpha, \beta)\in \varphi^{-1} U\times \varphi^{-1}V$, while the function   $\rho^{X}_{o}(a , b)$ is  bounded over all $(a,b)\in U\times V$.
Hence, there is an $M>0$ with $M^{-1}\leq J(a)J(b)\leq M$ for all $(a,b)\in (U\times V)\cap \Omega$.

Let $b\in V$ be such that $(a,b)\in \Omega$ for $\nu_{X}$-almost all $a\in U$ (since $\Omega$ has full measure, such   $b$ exists by Fubini's theorem).

Then $M^{-1}\leq J(a)J(b)\leq M$ for $\nu_{X}$-almost all $a\in U$.

Thus $$M^{-2}\leq J(a)/J(a')\leq M^{2}$$ for $\nu_{X}$-almost all $a,a'\in U$.

Thus $$M^{-2}\leq J(a)/J(a')\leq M^{2}$$ for $\nu_{X}$-almost all $a,a'\in U$.

Now, let $W_{1}, W_{2}\subset \partial X$ be closed neighborhoods whose complements contain closed sets with nonempty interiors and such that $W_{1}\cup W_{2}=\partial X$. By the above argument there is an $M>0$ such that $M^{-2}\leq J(a)/J(a')\leq M^{2}$ for almost all $(a,a'){\in}(W_{1}\times W_{1})\cup (W_{2}\times W_{2})$.
This implies that $J$ is essentially bounded on each of $W_{1}$ and $W_{2}$, and thus $J$ is essentially bounded on $\partial X$.
\end{proof}
We now return to the proof of Theorem \ref{singularityofmeasures}\begin{proof}[Proof of Theorem \ref{singularityofmeasures}]

Let $E$ be the quasiconformal constant for $\kappa$. Then for $\kappa$-a.e. $\zeta \in \partial X$,
\begin{equation}\label{qc}  E^{-1} \exp(-h\beta^{X}_{\zeta}(go,o))\leq \frac{dg\kappa}{d\kappa}(\zeta)\leq E \exp(-h\beta^{X}_{\zeta}(go,o)).
\end{equation}
Moreover, by (\ref{hypkernel}) we have

\begin{equation}\label{harmonic}
\frac{dg\nu_X}{d\nu_X}(\zeta)=K(g,\psi(\zeta))
\end{equation}

for $\nu_X$-a.e. $\zeta \in \partial X$.

By Lemma \ref{singularvsbounded}, if $\nu_X$ and $\kappa$ are not singular, they are absolutely continuous and there is a constant $C>1$  with

$$C^{-1}\leq \frac{d\kappa}{d\nu_X}\leq C$$ $\nu_{X}$-a.e.
Thus, since
$$\frac{d\kappa}{d\nu_X}\circ g^{-1}=\frac{dg\kappa}{dg\nu_X}$$
we have for every $g\in G$
$$C^{-1}\leq \frac{dg\kappa}{dg\nu_X}\leq C$$ $\nu_{X}$-a.e.

Consequently for every $g\in G$
$$\frac{dg\kappa/d\kappa}{dg\nu_{X}/d\nu_{X}}=\frac{dg\kappa}{dg\nu_{X}}\frac{d\nu_{X}}{d\kappa}\in [C^{-2},C^{2}]$$ $\nu_{X}$-a.e. so taking logarithms we get
$$|\log \frac{dg\kappa}{d\kappa}-\log \frac{dg\nu_{X}}{d\nu_X}|<C'$$ $\nu_{X}$-a.e. for a uniform constant $C'$.

By (\ref{qc}), and (\ref{harmonic}) this implies that there is $C''>0$ and a set $\Upsilon\subset \partial X$ of full $\nu_{X}$ and $\kappa$ measure with
\begin{equation}\label{essential}
|\sup_{\alpha\in \varphi^{-1}\Upsilon}\Delta(g,e,\alpha)-h\cdot \sup_{a\in \Upsilon}\beta^{X}_{a}(go,o)|<C''
\end{equation}

(this is well defined since for each $g$, the quantities $\Delta(g,e,\alpha)=-\log K(g,\alpha)$ and $\beta^{X}_{\zeta}(go,o)$ are bounded over $\alpha \in \partial_{\mathcal M} G$ and $\zeta \in \partial X$ respectively).

Since $\Upsilon$ has full $\nu_{X}$ measure {the set} $U=F^{-1}\Upsilon\subset \partial_{f}G$ has full $\nu_{f}$ measure.
Furthermore, $\nu_{f}$ has full support on ${\partial_fG}$ by
Theorem \ref{basicstationary}   {so} $\overline{U}=\partial_{f}G$.

{Thus} by continuity of $\varphi$ we have
\begin{equation}\label{closuregreen}
\sup_{\alpha\in \varphi^{-1}\Upsilon}\Delta(g,e,\alpha)=\sup_{\alpha\in \pi^{-1}U}\Delta(g,e,\alpha)=\sup_{\alpha\in \overline{\pi^{-1}U}}\Delta(g,e,\alpha).
\end{equation}
%Since $\Upsilon$ has full $\nu_{X}$ measure we know $U=F^{-1}\Upsilon$ has full $\nu_{f}$ measure.
%Also, $\nu_{f}$ has full support on $\Lambda(G)$ by
%Theorem \ref{basicstationary} and so $\overline{U}=\partial_{f}G$.
%Hence, $$\pi(\overline{\pi^{-1}U})
%=\overline{U}=\partial_{f}G.$$

Furthermore, since $\partial_{f} G=\pi(\overline{\pi^{-1}U})$ has more than two points, there is  $\delta>0$ such that for each $g\in  G$ there is an infinite sequence $z_n\in  G$ converging to some $\alpha \in \overline{\pi^{-1}U}$ with $\delta^{o}_{f}(g,z_n)> \delta$ for all $n$.

Consequently by Theorem \ref{relancona}
$$
d_{\mathcal G}(g,e)+d_{\mathcal G}(e,z_{n})-A(\delta) \leq d_{\mathcal G}(g,z_{n})
$$
and thus $$\Delta(g,e,\alpha)\geq d_{\mathcal G}(g,e)-A(\delta).$$

This implies
\begin{equation} \label{greenqr}
d_{\mathcal G}(g,e)-A(\delta)\leq \sup_{\alpha \in \overline{\pi^{-1}U}}\Delta(g,e,\alpha)\leq d_{\mathcal G}(g,e)
\end{equation} for all $g\in G$.

On the other hand, clearly
\begin{equation}\label{busbound}
\sup_{a\in \Upsilon}\beta^{X}_{a}(go,o)
\leq d_{X}(go,o)
\end{equation}

Putting together (\ref{essential}),  (\ref{closuregreen}), (\ref{greenqr}), and  (\ref{busbound})
we {obtain} a constant $D>0$ such that
$$d_{\mathcal G}(g,{e})-h\dot d_{X}(go,o)<D$$
for all $g\in G$.

%No abs. value since \ref{busbound} is only in one direction

Since $d_{\mathcal G}$ is quasi-isometric to the word metric $||.||$ this implies
\begin{equation}\label{fineq}
c'||g||-d_{X}(go,o)<D
\end{equation}
for all $g\in G$ and a constant $c''>0$.

On the other hand, since our word metric comes from a finite generating set it is clear that the orbit map is coarsely Lipschitz, i.e. there are constants $K_1$, $K_2$ with $$d_{X}(go,o)\leq K_{1}||g||+K_{2}$$
for all $g\in G$. 

This implies that the orbit map 
$G\to Go\subset X$ is a quasi-isometric embedding, which is impossible in the presence of parabolic elements \cite{Swenson}.

\end{proof}


\begin{thebibliography}{9}

\bibitem{Ancona} Ancona, A. (1988). Positive harmonic functions and hyperbolicity. In Potential Theory - Surveys and Problems (Prague, 1987). Lecture Notes in Math. 1344
1â€“23. Springer, Berlin.

\bibitem{AS} M. Anderson and R. Schoen. Positive harmonic functions on complete manifolds of negative curvature, Ann.
of Math. 121 (1985), no. 3, 429â€“461.

\bibitem{Babillot} M. Babillot. On the mixing property for hyperbolic systems. Israel J. Math. 129
(2002) 61â€“76.

\bibitem{BB} S. Blachere and S. Brofferio, Internal diffusion limited aggregation on discrete groups having exponential growth.
Probability Theory and Related Fields, 137 (2007), n. 3-4 323-343
\bibitem{BHM}
{S. Blachere, P. HaÃ¯ssinsky, and P. Mathieu,} Harmonic measures versus quasiconformal
measures for hyperbolic groups, Ann. Sci. Ã‰c. Norm. SupÃ©r. (4) 44 (2011), {683-721}.


\bibitem{Bowditch}
B. H. Bowditch, {\sl Relatively hyperbolic
groups,} Internat. J. Algebra Comput. 22 (2012), no. 3,
1250--1316.

\bibitem{Bo2} B. H. Bowditch, {\sl Convergence groups and configuration
spaces},  in ``Group theory down under'' (ed.\ J.Cossey,
C.F.Miller, W.D.Neumann, M.Shapiro), de Gruyter (1999) 23--54.

\bibitem{CDP}
M.Coornaert, T.Delzant, and A.Papadopoulos. Géométrie et théorie des groupes: les groupes hyperboliques de Gromov", Lecture Notes in Mathematics, vol. 1441,
Springer-Verlag, Berlin, 1990.

\bibitem{Coornaert}
M. Coornaert, Mesures de Patterson-Sullivan sur le bord dâ€™un espace hyperbolique au sens de Gromov,
Pacific J. Math. 159 (1993), {241-270}.

\bibitem{Dalbo} F. Dalâ€™bo, Topologie du feuilletage fortement stable. Ann. Inst. Fourier
(Grenoble) 50 no. 3, (2000) {981-993}.

\bibitem{Day}
 M. M. Day. Convolutions, means and spectra. Illinois J. Math. 8 (1964) 100-111 (3).

\bibitem{DGGP}
M. Dussaule, I. Gekhtman, V. Gerasimov and L. Potyagailo. Martin boundary of relatively hyperbolic groups with virtually abelian parabolic subgroups. In preparation.  

\bibitem{DKN}
Bertrand Deroin, Victor Kleptsyn, and Andres Navas, On the question of ergodicity for minimal group actions on the circle, Mosc. Math. J. 9 (2009), no. 2, 263-303.

\bibitem{Floyd} W. J. Floyd, Group completions and limit sets of Kleinian groups, Inventiones Math.
57, 1980, {205-218.}


\bibitem{Furman}
A. Furman, Random walks on groups and random transformations
Handbook of dynamical systems, Vol. 1A, 931 - 1014, North-Holland, Amsterdam, 2002.
\bibitem{Guivarch}Y. Guivarch, Sur la loi des grands nombres et le rayon spectral dâ€™une marche
aleatoire, Conference on Random Walks (Kleebach, 1979), Asterisque, vol. 74

\bibitem{Furstenberg}
 H. Furstenberg, Boundary theory and stochastic processes on homogeneous spaces, in
Proc. Sympos. Pure Math. 26, AMS, Providence R. I., 1973, {193-229.}

\bibitem{GMT}
V. Gadre, J. Maher, and G.Tiozzo. Word length statistics and Lyapunov exponents for Fuchsian groups with cusps, New York J. Math 21 (2015), 511-531 .

\bibitem{GM}
{F. W.} Gehring and G. J. Martin, Discrete convergence groups, Complex analysis, I (College Park, Md., 1985â€“86), Springer, Berlin, 1987, pp. {158-167.}

\bibitem{Gequid}
I. Gekhtman.
Equidistribution of closed geodesics along random walk trajectories with respect to the harmonic invariant measure. Preprint available at https://arxiv.org/abs/1711.04985


\bibitem{Ge1}
V. Gerasimov, {\sl Expansive Convergence Groups are Relatively
Hyperbolic}, GAFA {\bf 19} (2009) 137--169.

\bibitem{Gerasimov} V. Gerasimov. Floyd maps for relatively hyperbolic groups, Geom. Funct. Anal. (2012), no. 22, {1361-1399.}

\bibitem{GePoJEMS}  V. Gerasimov, L. Potyagailo, Quasi-isometric maps and Floyd boundaries of relatively hyperbolic groups, J. Eur. Math.
Soc. 15 (2013), no. 6, 2115-2137.


\bibitem{GePoGGD} V. Gerasimov, L. Potyagailo, Non-finitely generated relatively hyperbolic groups and Floyd quasiconvexity, Groups,
Geometry and Dynamics {9} (2015), {369-434}.
\bibitem{GePoCrelle}V. Gerasimov, L. Potyagailo, Quasiconvexity in the relatively hyperbolic groups, Journal fÂ¨ur die reine und angewandte
Mathematik (Crelle journal) 710 (2016), {95-135.}

\bibitem{G1} S. Gouezel, Local limit theorem for symmetric random walks in Gromov-hyperbolic groups, J. Amer.
Math. Soc., 27 (2014), {893-928}.
\bibitem{G2} S. Gouezel, Martin boundary of random walks with unbounded jumps in hyperbolic groups, Ann.
Proba., 43 (2015), {2374-2404}.
\bibitem{GL} S. Gouezel and S. Lalley, Random walks on co-compact Fuchsian groups, Ann. Sci. Ecole Norm. Â´
Sup. (4) 46 (2013), {129-173}.

\bibitem{Gromov} M. Gromov, {\sl Hyperbolic groups}, in: ``Essays in Group
Theory'' (ed. S.~M.~Gersten) M.S.R.I. Publications No.~8,
Springer-Verlag (1987) 75--263.

\bibitem{Hochman} Michael Hochman, A ratio ergodic theorem for multiparameter nonsingular
actions, J. Eur. Math. Soc. (JEMS) 12 (2010), no. 2, {365-383.}

\bibitem{Kaierg}
V. Kaimanovich, Ergodicity of harmonic invariant measures for the geodesic flow on hyperbolic
spaces, Journal fÃ¼r die reine und angewandte Mathematik (1994)
Volume: 455, page 57-104.

\bibitem{KaiStrip}
V. Kaimanovich, The Poisson formula for groups with hyperbolic properties,
Ann. of Math. (2) 152 (2000), no. 3, {659-692.}

\bibitem{Karl} A. Karlsson Free subgroups of groups with non-trivial Floyd boundary, Comm. Algebra, 31, (2003),
5361â€“5376.

\bibitem{Karlpoisson}
 A. Karlsson, Boundaries and random walks on finitely generated infinite groups, Arkiv f\"or Matematik, 41 (2003) {295-306.}

\bibitem{Kesten1}
H. Kesten. Symmetric random walks on groups. Trans. Amer. Math. Soc. 92 (1959) 336-354 (2, 3).

\bibitem{Kesten2}
 H. Kesten. Full Banach mean values on countable groups, Math. Scand. 7 (1959) 146-156 (3).

\bibitem{Maher-Tiozzo}
J. Maher and G. Tiozzo, Random walks on weakly hyperbolic groups, to appear in Journal fÃ¼r die reine und angewandte Mathematik. Preprint available at
https://arxiv.org/abs/1410.4173


\bibitem{MYJ}
K. Matsuzaki, Y. Yabuki, J. Jaerisch,
Normalizer, divergence type and Patterson measure for discrete groups of the Gromov hyperbolic space.
Preprint, https://arxiv.org/abs/1511.02664

\bibitem{PW} M. Picardello and W. Woess. Martin boundaries of Cartesian products of Markov chains. Nagoya Math. J.
Volume 128 (1992), 153-169.

\bibitem{Roblin}
 T. Roblin. Ergodicite et equidistribution en courbure negative. Memoire Soc. Math.
France, 95 (2003).

\bibitem{Sawyer}
S. Sawyer,  Martin boundaries and random walks.

http://math.wustl.edu/~sawyer/hmhandouts/martbrwf.pdf

\bibitem{Swenson}
E. Swenson, Quasi-convex groups of isometries of negatively curved spaces.
Topology and its Applications 110 (2001) 119–129.

\bibitem{Tiozzo}
G. Tiozzo, Sublinear deviation between geodesics and sample paths, Duke Math. J. 164 (2015), no. 3, 511-539.


\bibitem{Tu1}
  P. Tukia, {\sl Convergence groups and Gromov's metric
hyperbolic spaces} : New Zealand J.\ Math.\ \bf 23 \rm (1994)
157--187.

\bibitem{Varopoulos} N. Varopoulos. Theorie du potentiel sur des groupes et des varietes', C. R. Acad. Sci. Paris
Se'r. I 302 (1986) 203-205
\bibitem{Walters} P. Walters. An Introduction to Ergodic Theory. Springer 1982.

\bibitem{Woess} Woess, W. (2000). Random Walks on Infinite Graphs and Groups. Cambridge Tracts
in Mathematics 138. Cambridge Univ. Press, Cambridge.


\bibitem{Wo2} W. Woess, Lamplighters, Diestel-Leader graphs, random walks, and harmonic functions. Combin. Probab. Comput. 14 (2005), no. 3, {415-433.}

\bibitem{Ya}
A. Yaman, {\sl A topological characterisation of relatively
hyperbolic groups}, J.\ reine ang.\ Math. \bf 566 \rm(2004),
41--89.

\end{thebibliography}
\end{document}